\documentclass[12pt]{amsart}
\usepackage[utf8]{inputenc}
\usepackage{calligra}
\usepackage{mathrsfs}
\usepackage{delarray}
\usepackage{enumerate}
\usepackage[all,cmtip]{xy}
\usepackage{geometry}
\usepackage{todonotes}
\usepackage{comment}
\usepackage[backref]{hyperref}
\usepackage{xcolor}
\usepackage{float}
\usepackage{soul}

\newtheorem{theorem}{Theorem}[section]
\newtheorem{lemma}[theorem]{Lemma}
\newtheorem{proposition}[theorem]{Proposition}
\newtheorem{corollary}[theorem]{Corollary}
\theoremstyle{definition}
\newtheorem{definition}[theorem]{Definition}

\newtheorem{remark}[theorem]{Remark}

\newtheorem{example}[theorem]{Example}

 \newtheorem{theorem*}{Theorem}
 \newtheorem{corollary*}[theorem*]{Corollary}
 \newtheorem{proposition*}[theorem*]{Proposition}

\newcommand{\CC}{{\mathbb{C}}}
\newcommand{\QQ}{{\mathbb{Q}}}
\newcommand{\RR}{{\mathbb{R}}}

\newcommand{\ZZ}{{\mathbb{Z}}}

\newcommand{\set}[2]{\ensuremath{\left \{\,{#1}\,\middle|\, {#2}\,\right\}}}
\newcommand{\smvee}{\raise0.9ex\hbox{$\scriptscriptstyle\vee$}}

\newcommand{\Ss}[1]{\mathcal{O}_{#1}}

\newcommand{\Rs}{\tilde{X}}

\newcommand{\Sf}[1]{\mathcal{#1}}
\newcommand{\dimc}[1]{\dim_{\mathbb{C}}\left(#1\right)}

\DeclareMathOperator{\Exts}{\mathscr{E}\text{\kern -3pt {\calligra\large xt}}\,}
\DeclareMathOperator{\Homs}{\mathscr{H}\text{\kern -3pt {\calligra\large om}}\,}
\DeclareMathOperator{\Hom}{\text{Hom}\,}
\DeclareMathOperator{\Ext}{\text{Ext}\,}
\DeclareMathOperator{\sign}{sign}
\DeclareMathOperator{\rank}{rank}
\DeclareMathOperator{\trace}{Tr}
\DeclareMathOperator{\Ind}{Ind}
\DeclareMathOperator{\ind}{ind}

\newcommand{\GLn}[2][\CC]{\mathrm{GL}(#2,#1)}
\newcommand{\SLn}[2][\CC]{\mathrm{SL}(#2,#1)}
\newcommand{\SU}{\mathrm{SU(2)}}
\newcommand{\BG}[1][\Gamma]{\mathrm{B}#1}
\newcommand{\EG}[1][\Gamma]{\mathrm{E}#1}

\newcommand{\ccs}{Cheeger-Cheern-Simons}
\newcommand{\abs}[1]{\lvert #1\rvert}	

\begin{document}

\title[]{Cheeger-Chern-Simons classes of representations of finite subgroups of $\SLn{2}$ and the spectrum of rational double point singularities}

\author[J.~A.~Arciniega-Nevárez]{José Antonio Arciniega-Nevárez}
\address{División de Ingenierías, Campus Guanajuato, Universidad de Guanajuato\\
Av. Juárez No. 77, Zona Centro, Guanajuato, Gto., México, C.P. 36000\\ +524731020100-5924}
\email{ja.arciniega@ugto.mx}

\author[J.~L.~Cisneros-Molina]{Jos\'e Luis Cisneros-Molina}
\address{Instituto de Matem\'aticas, Unidad Cuernavaca\\ Universidad Nacional Aut\'onoma de M\'exico\\ Avenida Universidad s/n, Colonia Lomas
de Chamilpa\\ Cuernavaca, Morelos, Mexico.}
\email{jlcisneros@im.unam.mx}

\author[A.~Romano-Velázquez]{Agust\'in Romano-Vel\'azquez}
\address{Alfréd Rényi Institute of Mathematics, Hungarian Academy of Sciences, Reáltanoda Utca 13-15, H-1053, Budapest, Hungary.\newline Instituto de Matem\'aticas, Unidad Cuernavaca\\ Universidad Nacional Aut\'onoma de M\'exico\\ Avenida Universidad s/n, Colonia Lomas
de Chamilpa\\ Cuernavaca, Morelos, Mexico.}
\email{agustin@renyi.hu, agustin.romano@im.unam.mx}


\begin{abstract}
Let $L$ be a compact oriented $3$-manifold and $\rho\colon\pi_1(L)\to\GLn{n}$ a representation.
Evaluating the Cheeger-Chern-Simons class $\widehat{c}_{\rho,k}\in H^{2k-1}(L;\CC/\ZZ)$ of $\rho$ in homology classes $\nu\in H_{2k-1}(L;\ZZ)$
we get characteristic numbers that we call the $k$-th CCS-numbers of $\rho$. In Theorem~\ref{th:C2.APS} we prove that if $\rho$ is a topologically trivial representation, the second
CCS-number $\widehat{c}_{\rho,2}([L])$, where $[L]$ is the fundamental class of $L$, is given by the invariant $\tilde{\xi}_\rho(D)$ of the Dirac operator $D$ of $L$ twisted by $\rho$,
defined by Atiyah, Patodi and Singer \cite{Atiyah-Patodi-Singer:SARGIII}. If $L$ is a rational homology sphere we also give a formula for $\widehat{c}_{\rho,2}([L])$ for any representation $\rho$ in terms of $\tilde{\xi}$-invariants
of $D$.
Given a topologically trivial representation $\rho\colon\pi_1(L)\to\GLn{n}$ we construct an element $\langle L,\rho\rangle$ in $K_3(\CC)$, the $3$rd algebraic K-theory group of the complex numbers.
When $L$ is an integral homology sphere we prove that $\langle L,\rho\rangle$ coincides with the element in $K_3(\CC)$ constructed by Jones and Westbury in \cite{Jones-Westbury:AKTHSEI} and that Theorem~\ref{th:C2.APS} generalizes \cite[Theorem~A]{Jones-Westbury:AKTHSEI}.
For rational homology spheres of the form $L=\mathbb{S}^3/\Gamma$, where $\Gamma$ is a finite subgroup of $\SU$, we compute the first and second
CCS-numbers of all the irreducible representations of $\Gamma$, with respect to the generators of $H_{2k-1}(L;\ZZ)$ with $i=1,2$. Using these CCS-numbers, we recover the spectrum of all rational
double point singularities, which is an invariant of hypersurface singularities defined by Steenbrink in \cite{Steenbrink:SHS}.
Motivated by this result, we define the \textit{topological spectrum} for rational surface singularities and Gorenstein singularities.
Given a normal surface singularity $(X,x)$ with link $L$ a rational homology sphere, we show how to compute the invariant $\tilde{\xi}_\rho(D)$ for the Dirac operator of $L$ using a resolution or a smoothing of $(X,x)$.
\end{abstract}
\maketitle

\section{Introduction}

Cheeger-Chern-Simons classes $\widehat{c}_k(E,\nabla)\in H^{2k-1}(M;\CC/\ZZ)$ are secondary charac\-te\-ris\-tic classes of a vector bundle $E$ over a smooth manifold $M$ with a \textit{flat} connection $\nabla$, they were defined by Cheeger and Simons in \cite{zbMATH04007261}. Given a representation $\rho\colon M\to \GLn{n}$ of the fundamental group of $M$,
there is an associated flat vector bundle $V_\rho$ with a canonical flat connection $\nabla_\rho$, following \cite[\S8]{zbMATH04007261} the Cheeger-Chern-Simons classes of the representation $\rho$ are defined by $\widehat{c}_{\rho,k}=\widehat{c}_{k}(V_\rho,\nabla_\rho)\in H^{2k-1}(M;\CC/\ZZ)$. Evaluating $\widehat{c}_{\rho,k}$ on a homology class $\nu\in H_{2k-1}(M;\ZZ)$ one gets characteristic numbers $\widehat{c}_{\rho,k}(\nu)$ which we call CCS-numbers.

In the present article we restrict to the case of a compact oriented $3$-manifold $L$.
Using the Index Theorem for flat bundles by Atiyah, Patodi and Singer \cite[(5.3)]{Atiyah-Patodi-Singer:SARGIII}, in Theorem~\ref{th:C2.APS} we prove that for topologically trivial representations
$\rho\colon\pi_1(L)\to\GLn{n}$ the second CCS-number $\widehat{c}_{\rho,2}([L])$, where $[L]\in H_3(L;\ZZ)$ is the fundamental class of $L$, agrees with the invariant
$\tilde{\xi}_\rho(D)$ of the Dirac operator $D$ of $L$ twisted by the representation $\rho$, defined also in \cite[(3.2)]{Atiyah-Patodi-Singer:SARGIII}.
For the case when $L$ is a rational homology $3$-sphere, we also give a formula for the second CCS-number $\widehat{c}_{\rho,2}([L])$ of any representation $\rho$.

A representation $\rho\colon\pi_1(L) \to \GLn{n}$ is topologically trivial if and only if it factors through $\SLn{n}$.
Hence, it induces a map $f_n\colon L\to\BG[\SLn{n}]^d$ between $L$ and the classifiying space of $\SLn{n}$ with discrete topology,
and a homomorphism in homology $(f_n)_*\colon H_3(L;\ZZ)\to H_3(\BG[\SLn{n}]^d;\ZZ)$.
The image of the fundamental class of $L$ gives an element $\langle L,\rho\rangle=f_*([L])\in H_3(\BG[\SLn{n}]^d;\ZZ)$.
The second Cheeger-Chern-Simons class $\widehat{c}_{\rho,2}$ of $\rho$ is the pull-back by $f_n$ of a class $\widehat{\widehat{c}}_2\in H^3(\BG[\SLn{n}]^d;\CC/\ZZ)$
(the restriction to $\BG[\SLn{n}]^d$ of the $k$-th universal Cheeger-Chern-Simons class $\widehat{c}_2\in H^3(\BG[\GLn{n}]^d;\CC/\ZZ)$, see Remark~\ref{rem:uccs}),
thus, evaluating $\widehat{\widehat{c}}_2$ on $\langle L,\rho\rangle$ we recover the second CCS-number $\widehat{c}_{\rho,2}([L])=\tilde{\xi}_\rho(D)$.
We can identify the element $\langle L,\rho\rangle$ with an element in the $3$rd algebraic K-theory group of the complex numbers via the isomorphism
$K_3(\CC)=H_3(\SLn{n}^d;\ZZ)$ for $n\geq3$.
On the other hand, given an integral homology sphere $\Sigma$ and a representation $\rho\colon \pi_1(\Sigma)\to\GLn{n}$, in \cite{Jones-Westbury:AKTHSEI}
Jones and Westbury construct an element $[\Sigma,\rho]\in K_3(\CC)$ whose image under a regulator homomorphism $e\colon K_3(\CC)\to\CC/\ZZ$ gives the invariant $\tilde{\xi}_\rho(D)$.
We prove that for an integral homology sphere $\Sigma$ we have $\langle \Sigma,\rho\rangle=[\Sigma,\rho]\in K_3(\CC)$ and that Theorem~\ref{th:C2.APS} is a generalization
of \cite[Theorem~A]{Jones-Westbury:AKTHSEI}.

Next, we focus on the rational homology $3$-spheres of the form $L=\mathbb{S}^3/\Gamma$, where $\Gamma$ is a finite subgroup of $\SU$. We compute the first and second
CCS-numbers of all the irreducible representations of $\Gamma$, with respect to the generators of $H_1(L;\ZZ)$ and the fundamental class $[L]$ respectively.

Quoting Arnol'd in \cite{Arnold:SPST} ``Much progress in singularity theory of differentiable maps is based on empirical
data. Some of these empirical facts later become theorems.'' Our next result is an empirical fact about rational double point singularities, that we hope it will become a theorem.
A rational double point singularity is a quotient singularity of the form $(\CC^2/\Gamma,0)$, where $\Gamma$ is a finite subgroup of $\SLn{2}$. Its link is the rational homology sphere
$L=\mathbb{S}^3/\Gamma$ mentioned above.
In \cite{Steenbrink:SHS} Steenbrink defined an invariant for hypersurface singularities using the mixed Hodge structure on the cohomology of the Milnor fibre together with the monodromy. It is called the \textit{spectrum} and it is a set of rational numbers (not necessarily all different). In Theorem~\ref{th:C1C2SonSpf} we give a recipe to recover the spectrum of a rational double point singularity from the first CCS-numbers of the irreducible representations of $\Gamma$ and a new invariant $\Xi_{\alpha_{\mathrm{Nat}}}(X,\Rs)$ which is a multiple of the second CCS-number of the natural representation given by the coefficients of the fundamental cycle of the minimal resolution of the singularity. Motivated by Theorem~\ref{th:C1C2SonSpf} we define a new invariant, the \textit{topological spectrum}, for rational surface singularities and Gorenstein singularities,
using respectively, the Cheeger-Chern-Simons classes of \textit{special} \cite{zbMATH03997956} and \textit{cohomologically special} \cite{BoRo} representations of the fundametal group of its link.

Finally, given a normal surface singularity $(X,x)$ with link $L$ a rational homology sphere, using a result by Atiyah, Patodi and Singer \cite[pp.~415]{zbMATH03491931} we show how to compute $\tilde{\xi}_\rho(D)$ for the Dirac operator of $L$ using a resolution or a smoothing of $(X,x)$.

\section{Characteristic classes of representations}
\label{Sec:Pre}
In this section we recall basics on classifying spaces, we give the definition of the Chern classes of a representation given by Atiyah in \cite{Atiyah-PMIHES} and we define
the Chern and Cheeger-Chern-Simons classes of a representation of the fundamental group of a manifold. 
We assume basic familiarity with group homology, representation theory, fibre bundles and Chern-Weil theory, 
see~\cite{zbMATH03935317,Husemoller:Bundles,Chern:TDG} for more details.

\subsection{Principal (flat) bundles}
In this section we recall some basic results about principal bundles, see \cite[Chapter~4]{Husemoller:Bundles} for details. For any topological group $G$ there exist spaces (unique up to homotopy) $\EG[G]$ and $\BG[G]$; and a map (unique up to homotopy) $p\colon \EG[G] \to \BG[G]$.
The space $\EG[G]$ is contractible and the map $p$ is a $G$-principal bundle, that is, $\EG[G]$ admits a free $G$-action, so we can think of $\BG[G]$ as the orbit space $\EG[G]/G$ and $p$ as the projection. Therefore, by the long exact homotopy sequence, we have $\pi_{n+1}(\BG[G])\cong \pi_{n}(G)$.
The space $\BG[G]$ is called the \emph{classifying space of $G$} and the principal $G$-bundle $p$ the \emph{universal principal $G$-bundle} because any principal $G$-bundle over a paracompact space $X$ is the pull-back of $p$ by a \emph{classifying map} $X\to \BG[G]$.
The construction of $BG$ is functorial in $G$.

Thus, there is a one-to-one correspondence between the set $\text{Prin}_{\GLn{n}}(X)$ of isomorphism classes of principal $\GLn{n}$-bundles over $X$
and the set $[X, \BG[\GLn{n}]]$ of homotopy classes of maps from $X$ to $\BG[\GLn{n}]$:
\begin{equation*}
    \text{Prin}_{\GLn{n}}(X) \Longleftrightarrow [X, \BG[\GLn{n}]].
\end{equation*}
There is also a one-to-one correspondence between principal $\GLn{n}$-bundles and complex vector bundles of rank $n$ over $X$:
to any principal $\GLn{n}$-bundle corresponds the complex vector bundle given by the standard action of $\GLn{n}$ on $\CC^n$ and, given a complex vector bundle of rank $n$
its frame bundle is a principal $\GLn{n}$-bundle. So, if we denote by $\text{Vect}_{n,\CC}(X)$ 
the set of isomorphism classes of complex vector bundles of rank $n$ over $X$, then we can write the previous correspondence as
\begin{equation}\label{eq:Vect.class.map}
\text{Vect}_{n,\CC}(X) \Longleftrightarrow [X, \BG[\GLn{n}]].
\end{equation}

We denote by $G^d$ the same group $G$ but with the discrete topology. Note that $BG^d$ is an Eilenberg-MacLane space of type $K(G^d,1)$.
The natural continuous map $\iota \colon G^d \to G$ given by the identity, induces a natural map between the classifying spaces $\iota \colon \BG[G]^d \to \BG[G]$.
Let $M$ be a compact manifold. 
A principal $G$-bundle $G\to P\to M$ with classifying map $\bar{f}\colon M\to\BG[G]$ is called \emph{flat} if $f$ factors up to homotopy as
\begin{equation}\label{eq:flat.fact}
\xymatrix{
&\BG[G]^d \ar[d]^\iota \\
M \ar[r]^{\bar{f}} \ar@{.>}[ru]^{f} & \BG[G].
}
\end{equation}
Such a factorization has the effect of reducing
the structure group of the bundle to a discrete group, so that any associated vector bundle has a \emph{flat connection}, i.~e., with curvature zero, see \cite[\S4]{Oprea-Tanre:FCBPBCMD}.
So, if we denote by $\text{Flat Prin}_{\GLn{n}}(M)$ (respectively $\text{Flat Vect}_{n,\CC}(M)$)
the set of isomorphism classes of flat principal (flat rank $n$ complex vector) bundles over $X$, then
one has the following one-to-one correspondences
\begin{equation}\label{eq:Vect.Class}
\begin{gathered}
\text{Flat Prin}_{\GLn{n}}(M) \Longleftrightarrow [M, \BG[\GLn{n}^{d}]],\\
\text{Flat Vect}_{n,\CC}(M) \Longleftrightarrow [M, \BG[\GLn{n}^d]] \Longleftrightarrow \Hom(\pi_1(M),\GLn{n}^d).
\end{gathered}
\end{equation}

\begin{remark}\label{rem:universal} 
Let $\mathrm{V}_n(\CC^K)$ be the \textit{Stiefel manifold} of $n$-frames in $\CC^K$ and $\mathrm{G}_n(\CC^K)$ the \textit{Grassmannian} of $n$-planes in $\CC^K$.
Set $\mathrm{V}_n=\mathrm{V}_n(\CC^\infty)=\varinjlim \mathrm{V}_n(\CC^K)$ and $\mathrm{G}_n=\mathrm{G}_n(\CC^\infty)=\varinjlim \mathrm{G}_n(\CC^K)$.
For $G=\GLn{n}$, the classifying space $\BG[G]$ is $\mathrm{G}_n$ and the universal principal $G$-bundle is the \textit{Stiefel bundle} $\mathrm{V}_n\to\mathrm{G}_n$.
The \emph{universal vector bundle} is the canonical complex bundle $\gamma^n$ of rank $n$ over $\mathrm{G}_n$ \cite[Theorem~8-6.1, Theorem~3-7.2]{Husemoller:Bundles}.
\end{remark}

\begin{remark}\label{rem:universal.connection}
In \cite[Theorem~1]{zbMATH03186497} Narasimhan and Ramanan proved that the \textit{Stiefel bundle} $\mathrm{V}_n(\CC^K)\to \mathrm{G}_n(\CC^K)$ has a connection $\vartheta$ which makes it 
an \textit{$m$-classifying} principal $\GLn{n}$-bundle when $K=(m+1)(2m+1)n^2$, that is, any
principal $G$-bundle over a manifold $M$ with a connection $\theta$ with $\dim M\leq m$, admits a connection preserving bundle map to $\mathrm{V}_n(\CC^K)$, and for any two such 
connection preserving bundle morphisms, the corresponding maps $f_1,f_2\colon M\to \mathrm{G}_n(\CC^K)$ are smoothly homotopic. The connections on the Stiefel bundles with different $K$ are
compatible, so we have a \textit{universal connection} $\vartheta$ in the universal bundle $\mathrm{V}_n\to\mathrm{G}_n$. This defines a universal connection $\nabla_\mathrm{univ}$ on the
universal vector bundle $\gamma^n\to \mathrm{G}_n$. We denote it by $(\gamma^n,\BG[\GLn{n}],\nabla_\mathrm{univ})$.
Hence the pull-back (with connection) of $\gamma^n$ by the map $\iota \colon \BG[\GLn{n}]^d \to \BG[\GLn{n}]=\mathrm{G}_n$ is a \emph{universal flat vector bundle} with \textit{universal connection}. We denote it by $(\gamma^n_d,\BG[\GLn{n}]^d,\nabla_\mathrm{univ}^d)=(\iota^*\gamma^n,\BG[\GLn{n}]^d,\iota^*\nabla_\mathrm{univ})$.
\end{remark}

\subsection{The Chern class of a representation}\label{ssec:cc.rho}
Let $\Gamma$ be a discrete group.
Let $\rho \colon \Gamma \to \GLn{n}$
be a representation. Such a representation induces a map between the classifying spaces
\begin{equation*}
\BG[\rho] \colon \BG[\Gamma] \to \BG[\GLn{n}],
\end{equation*}
which by \eqref{eq:Vect.class.map} induces a complex vector bundle $V_\rho=\EG[\Gamma]\times_\rho\CC^n\to \BG[\Gamma]$, where $\EG[\Gamma]\times_\rho\CC^n$ is the quotient of $\EG[\Gamma]\times\CC^n$ by the right action of $\Gamma$ given by $(e,v)g=(eg,\rho(g)^{-1}v)$.

The \emph{$k$-th Chern class of $\rho$} is the $k$-th Chern class of the vector bundle $V_\rho$, that is,
\begin{equation*}
 c_{k}(\rho)=c_k(V_\rho) \in H^{2k}(\BG[\Gamma];\ZZ),
\end{equation*}
(for $\Gamma$ finite see~\cite{Atiyah-PMIHES}, for the general case see~\cite{zbMATH03868233}).

For $\Gamma$ a finite group, the Chern class of a representation $\rho$ is characterized by the following three properties (see \cite{Kroll-AACOCCOFGR}\cite[Theorem~5]{Kroll:CCFGR}):
\begin{itemize}
    \item If $h\colon \Gamma' \to \Gamma$ is a group morphism, then
    \begin{equation*}
        c_{k}(h^{!}\rho) = h^{*}c_{k}(\rho),
    \end{equation*}
where $h^!$ and $h^*$ are the natural pull-back maps in representation theory and cohomology.
    \item If the total Chern class is $c(\rho)=1+c_{1}(\rho)+\dots+c_{n}(\rho)$, then
    \begin{equation*}
        c(\rho_1 \oplus \rho_2) = c(\rho_1) \cdot c(\rho_2).
    \end{equation*}
    \item The homomorphism
    \begin{equation*}
        c_1 \colon \Hom(\Gamma, U(1)) \to H^2(\BG[\Gamma];\ZZ),\qquad \rho\mapsto c_1(\rho),
    \end{equation*}
    is an isomorphism.
\end{itemize}
\begin{remark}
\label{rem:C1Det}
For any representation $\rho \colon \Gamma \to \GLn{n}$, the first Chern class satisfies
\begin{equation*}
    c_1(\rho)=c_1(\det(\rho)),
\end{equation*}
where $\det\colon \GLn{n} \to \GLn{1}$ is the determinant homomorphism and $\det(\rho)=\det\circ\rho$ (for a finite group see~\cite[Appendix (7)]{Atiyah-PMIHES}, for the general case see~\cite{zbMATH03868233}).
\end{remark}

\subsection{Cheeger-Chern-Simons classes}
\label{subsec:CCS}

Let $M$ be a smooth manifold. Let $(E,M,\nabla)$ be a vector bundle of rank $n$ over $M$ with a connection $\nabla$ and let $\Theta$ be its curvature.
Chern-Weil theory allows us to construct characteristic classes of the vector bundle $E$ using the curvature of a connection \cite{Chern:TDG}.

Let $\Lambda=\ZZ$ or $\QQ$.
Let $\Omega^k(M)$ be the group of smooth complex-valued differential $k$-forms on $M$, $\Omega_{cl}^k(M)$ the subgroup of closed $k$-forms and $\Omega^k_{cl}(M;\Lambda)$ the subgroup of closed $k$-forms with periods in $\Lambda$, i.~e., $\Omega^k_{cl}(M;\Lambda)=\ker\bigl(\Omega^k_{cl}(M)\to H^k(M;\CC/\Lambda)\bigr)$.
Let $\mathcal{M}_n\cong\CC^{n^2}$ denote the space of $n\times n$ matrices. Let $P_k\colon\mathcal{M}_n\to\CC$ be a symmetric homogeneous polynomial of degree $k$ with coefficients in $\Lambda$, which is invariant, i.~e., $P_k(A)=P_k(gAg^{-1})$ for all $A\in\mathcal{M}_n$ and all $g\in\GLn{n}$.
Denote by $I^k(\GLn{n})$ the set of invariant polynomials and set $I(\GLn{n})=\bigoplus_{k=0}^\infty I^k(\GLn{n})$ which is a commutative algebra.
Evaluating $P_k$ on the curvature $\Theta$ we get the $2k$-form $P_k(E,\nabla):=P_k(\Theta)$ which is closed and its de Rham cohomology class $P_k(E)=[P_k(E,\nabla)]\in H^{2k}(M;\CC)$ does not depend on the choice of connection $\nabla$. This defines the Weil homomorphism $I(\GLn{n})\to H^\bullet(M;\CC)$ which assigns to an invariant polynomial a cohomology class of $M$.
It is an algebra homomorphism.

For instance, using the invariant symmetric homogeneous polynomial $C_k(A)$ of degree $k$ given by the relation
\begin{equation}\label{eq:Chern-poly}
 \det(A+tI)=\sum_{k=0}^nC_k(A)t^{n-k},
\end{equation}
one gets the \emph{Chern forms} $c_k(E,\nabla)=C_k(\Theta)\in\Omega^{2k}_{cl}(M;\ZZ)$ for $0\leq k\leq n=\rank E$.
Its de Rham cohomology class $c_k^{dR}(E)=[c_k(E,\nabla)]\in H^{2k}(M;\CC)$ is the image of the $k$-th Chern class $c_k(E)$ of $E$ in the exact sequence 
\begin{equation}\label{eq:coef.CheegerSimons2}
 \begin{aligned}
   \dots\to H^{2k-1}(M;\CC/\ZZ) \xrightarrow{q} H^{2k}(M;\ZZ) &\xrightarrow{r} H^{2k}(M;\CC) \xrightarrow{p_\ZZ}\cdots\\
   c_k(E) &\mapsto c_k^{dR}(E)
\end{aligned}   
\end{equation}
induced by the short exact sequence of coefficients
\begin{equation}
    \label{eq:coef.CheegerSimons1}
    0 \to \ZZ \to \CC \to \CC/\ZZ \to 0.
\end{equation}

More generally, in place of an invariant polynomial, one can use an invariant power series of the form $P=P_0+P_1+P_2+\cdots$,
where each $P_k$ is an invariant homogeneous polynomial of degree $k$. 
For example, the invariant formal power series 
\begin{equation}\label{eq:Chern.power}
 ch(A)=\trace(e^{A/2\pi i}),
\end{equation}
where $\trace$ is the trace, gives $ch(E,\nabla)\in\Omega^{\mathrm{even}}(M)$ representing the \emph{Chern character} $ch(E)\in H^{\mathrm{even}}(M;\QQ)$, which can be writen in terms of the Chern classes of $E$ as
\begin{equation*}
 ch(E)=n+c_1(E)+\frac{1}{2}\bigl(c_1(E)^2-2c_2(E)\bigr)+\frac{1}{6}\bigl(c_1(E)^3-3c_1(E)c_2(E)+3c_3(E)\bigr)+\cdots.
\end{equation*}

In \cite{zbMATH04007261} Cheeger and Simons defined the ring of differential characters of $M$ and they constructed a lift of the Weil homomorfism to get secondary characteristic classes.

Let $C_k(M;\ZZ)\supset Z_k(M;\ZZ)$ be the group of smooth singular $k$-chains and $k$-cycles in $M$ and $\partial\colon C_k(M;\ZZ)\to C_{k+1}(M;\ZZ)$ the boundary operator.

The group of \emph{differential characters} of degree\footnote{It is convenient to shift the degree by +1 as compared to the original definition in \cite{zbMATH04007261}.}  $k$ is defined as
\begin{equation*}
\widehat{H}^k(M;\CC/\Lambda)=\left\{\psi\in \Hom(Z_{k-1}(M;\ZZ),\CC/\Lambda)\,\Bigl|\,\psi(\partial(a))=\int_a \omega_\psi \mod\Lambda\Bigr.\right\}.
\end{equation*}
The form $\omega_\psi$ is uniquely determined by $\psi$, we have that $\omega_\psi\in\Omega^k_{cl}(M;\Lambda)$. 
Let $r$ be the natural map $r\colon H^k(M;\Lambda)\to H^k(M;\CC)$ and given $\omega\in\Omega^k_{cl}(M;\Lambda)$ denote by $[\omega]$ its de Rham class.
Set
\begin{equation*}
 R^k(M;\Lambda)=\{(\omega,u)\in\Omega^k_{cl}(M;\Lambda)\times H^k(M;\Lambda)\,|\, r(u)=[\omega]\}.
\end{equation*}
There are natural exact sequences (see \cite[Theorem~1.1]{zbMATH04007261}):
\begin{gather}
 0\to H^{k-1}(M;\CC/\Lambda)\to \widehat{H}^k(M;\CC/\Lambda)\xrightarrow{\delta_1} \Omega^{k}_{cl}(M;\Lambda)\to0,\label{eq:ES1}\\
 0\to\Omega^{k-1}(M)/\Omega^{k-1}_{cl}(M;\Lambda)\to\widehat{H}^k(M;\CC/\Lambda)\xrightarrow{\delta_2}H^k(M;\Lambda)\to0.\label{eq:ES2}\\
 0\to H^{k-1}(M;\CC)/r(H^k(M;\Lambda))\to\widehat{H}^k(M;\CC/\Lambda)\xrightarrow{(\delta_1,\delta_2)} R^k(M;\Lambda)\to0.\label{eq:ES3}
\end{gather}
Given $\psi\in\widehat{H}^k(M;\CC/\Lambda)$ we have that $r(\delta_2(\psi))=[\delta_1(\psi)]$.
In particular, by \eqref{eq:ES3} if $H^{k-1}(M;\CC)=0$ then $\psi$ is determined uniquely by $(\delta_1(\psi),\delta_2(\psi))$.
In \cite{zbMATH04007261} a graded ring structure is defined in $\widehat{H}^\bullet(M;\CC/\Lambda)=\bigoplus_{k=0}^{\dim M}\widehat{H}^k(M;\CC/\Lambda)$
so that $\delta_1$ and $\delta_2$ are ring homomorphisms.

Let  $(\gamma^n,\BG[\GLn{n}],\nabla_\mathrm{univ})$ be the universal vector bundle with universal connection $\nabla_\mathrm{univ}$ (see Remark~\ref{rem:universal.connection}).
Let $P_k$ be an invariant polynomial, and $P_k(\gamma^n,\nabla_\mathrm{univ})\in\Omega^{2k}_{cl}(\BG[\GLn{n}];\Lambda)$ the form with periods in $\Lambda$ representing the characteristic
class $P_k(\gamma^n)\in H^{2k}(\BG[\GLn{n}];\Lambda)$.
Since $H^\mathrm{odd}(\BG[\GLn{n}];\CC)=0$,
considering \eqref{eq:ES3} for $\BG[\GLn{n}]$,
the element 
$(P_k(\gamma^n,\nabla_\mathrm{univ}),P_k(\gamma^n))\in R^{2k}(\BG[\GLn{n}];\Lambda)$ determines uniquely a differential character
$\widehat{P}_k=\widehat{P}_k(\gamma^n,\nabla_\mathrm{univ})\in \widehat{H}^{2k}(\BG[\GLn{n}];\CC/\Lambda)$, which we call the \textit{universal differential character}
defined by $P_k$.
Let $(E,M,\nabla)$ be a vector bundle with a connection $\nabla$. Let $\bar{f}\colon M\to \BG[\GLn{n}]$ be the classifying map with connection of $(E,M,\nabla)$.
Then $\widehat{P}_k(E,\nabla)=\bar{f}^*(\widehat{P}_k)\in \widehat{H}^{2k}(M;\CC/\Lambda)$ is a differential character which is
natural, and such that
\begin{equation}\label{eq:CCS.C}
 \delta_1(\widehat{P}_k(E,\nabla))=P_k(E,\nabla),\qquad \delta_2(\widehat{P}_k(E,\nabla))=P_k(E).
\end{equation}
Notice that the classes $\widehat{P}_k(E,\nabla)$ depend on the connection $\nabla$. If $\nabla_0$ and $\nabla_1$ are two connections on $E$,  we have
\begin{equation*}
\delta_2(\widehat{P}_k(E,\nabla_1)-\widehat{P}_k(E,\nabla_0))=\delta_2(\widehat{P}_k(E,\nabla_1))-\delta_2(\widehat{P}_k(E,\nabla_0))=0,
\end{equation*}
thus, by \eqref{eq:ES2} the difference of the characters must be the reduction of a differential form $TP_k(\nabla_1,\nabla_0)\in\Omega^{k-1}(M)$ modulo $\Lambda$. Such form can be computed using the Chern-Simons construction \cite{Chern-Simons:CFGI}. 
Let $\nabla_t$ be a smooth curve of connections joining $\nabla_0$ and $\nabla_1$ (e.~g., $\nabla_t=t\nabla_1+(1-t)\nabla_0$). Then $\nabla_t$ defines a connection $\widetilde{\nabla}$ on the
bundle $E\times I$. Let $\widetilde{\Theta}$ be the curvature of $\widetilde{\nabla}$. We can construct the characteristic form $P_k(E\times I,\widetilde{\nabla})=P_k(\widetilde{\Theta})\in\Omega^{2k}(M\times I)$ and the secondary \emph{Chern-Simons form}
\begin{equation}\label{eq:CS.form}
TP_k(\nabla_1,\nabla_0)=\pi_*(P_k(E\times I,\widetilde{\nabla}))\in\Omega^{2k-1}(M),
\end{equation}
where $\pi\colon M\times I\to M$ is the projection and $\pi_*\colon\Omega^{2k}(M\times I)\to\Omega^{2k-1}(M)$ is integration along the fibres (see \cite[Remark~19-1.8]{Husemoller:Bundles}). Hence
\begin{equation*}
\bigl(\widehat{P}_k(E,\nabla_1)-\widehat{P}_k(E,\nabla_0)\bigr)(a)=\int_aTP_k(\nabla_1,\nabla_0)\mod\Lambda,\quad a\in Z_{k-1}(M;\ZZ).
\end{equation*}
Since $TP_k(\nabla_1,\nabla_0)$ is closed we get a class $[TP_k(\nabla_1,\nabla_0)]\in H^{2k-1}(M;\CC)$ which is independent of the choice of path of connections $\nabla_t$.
If both connections $\nabla_0$ and $\nabla_1$ are flat, by \eqref{eq:ES1} $\widehat{P}_k(E,\nabla_0),\widehat{P}_k(E,\nabla_1)\in H^{2k-1}(M;\CC/\Lambda)$ and in this case we have
\begin{equation}\label{eq:CS.class.char}
\widehat{P}_k(E,\nabla_1)-\widehat{P}_k(E,\nabla_0)=p_\Lambda\bigl([TP_k(\nabla_1,\nabla_0)]\bigr),
\end{equation}
where $p_\Lambda\colon H^{2k-1}(M;\CC)\to H^{2k-1}(M;\CC/\Lambda)$ is the natural homomorphism.

Taking $P_k$ to be the polynomials $C_k$ defined by \eqref{eq:Chern-poly} we get the \emph{Chern differential characters} $\widehat{c}_k(E,\nabla)\in\widehat{H}^{2k}(M;\CC/\ZZ)$. 
On the other hand, using the power series given in \eqref{eq:Chern.power} we get the differential character $\widehat{ch}(E,\nabla)\in\widehat{H}^{\mathrm{even}}(M;\CC/\QQ)$ which can be written in
terms of $\widehat{c}_k(E,\nabla)$ (we drop $(E,\nabla)$ from the formula), see \cite[(4.10)]{zbMATH04007261}
\begin{equation*}
 \widehat{ch}(E,\nabla)=n+\widehat{c}_1+\frac{1}{2}\bigl(\widehat{c}_1*\widehat{c}_1-2\widehat{c}_2\bigr)+\frac{1}{6}\bigl(\widehat{c}_1*\widehat{c}_1*\widehat{c}_1-3\widehat{c}_1*\widehat{c}_2+3\widehat{c}_3\bigr)+\cdots,
\end{equation*}
where $*$ is the product in $\widehat{H}^{\bullet}(M;\CC/\QQ)$. We denote by $\widehat{ch}_k(E)$ the component of $\widehat{ch}(E)$ of degree $2k$, hence we have
\begin{align}
 \widehat{ch}_0(E,\nabla)&=n,\qquad \widehat{ch}_1(E,\nabla)=\widehat{c}_1(E,\nabla),\notag\\ 
 \widehat{ch}_2(E,\nabla)&=\frac{1}{2}\bigl(\widehat{c}_1(E,\nabla)*\widehat{c}_1(E,\nabla)-2\widehat{c}_2(E,\nabla)\bigr),\dots\label{eq:ch2}
\end{align}

\subsection{Cheeger-Chern-Simons classes of a representation}

Consider a representation $\rho\colon\pi_1(M)\to\GLn{n}$. Associated to $\rho$ we get the \emph{flat} $\GLn{n}$-bundle $V_\rho=\widetilde{M}\times_\rho\CC^n\to M$, where $\widetilde{M}$ is the
universal cover of $M$ and $\widetilde{M}\times_\rho\CC^n$ is the quotient of $\widetilde{M}\times\CC^n$ by the action of $\pi_1(M)$.
Since $V_\rho$ admits a flat connection $\nabla_\rho$, we have that $c_k(V_\rho,\nabla_\rho)=0$ and $c_k^{dR}(V_\rho)=0$, which imply that the Chern class $c_k(V_\rho)$ is in the image of the homomorphism from $q\colon H^{2k-1}(M;\CC/\ZZ)\to H^{2k}(M;\ZZ)$ in \eqref{eq:coef.CheegerSimons2}. Also by \eqref{eq:CCS.C} and \eqref{eq:ES1} we have that
\begin{equation}\label{eq:CCS.rep}
 \widehat{c}_k(V_\rho,\nabla_\rho)\in H^{2k-1}(M;\CC/\ZZ),
\end{equation}
and it is a functorial lifting of the Chern class $c_k(V_\rho)$ by $q$.
We define, respectively, the  \emph{Chern and Cheeger-Chern-Simons classes} of the representation $\rho$ by 
\begin{equation*}
 c_{\rho,k}=c_k(V_\rho)\in H^{2k}(M;\ZZ),\qquad \widehat{c}_{\rho,k}=\widehat{c}_{k}(V_\rho,\nabla_\rho)\in H^{2k-1}(M;\CC/\ZZ).
\end{equation*}

We set $\widehat{c}_\rho=1+\widehat{c}_{\rho,1}+\dots+\widehat{c}_{\rho,n}\in \widehat{H}^{\mathrm{even}}(M;\CC/\ZZ)$. If $\rho\colon\pi_1(M)\to\GLn{n}$ and $\tau\colon\pi_1(M)\to\GLn{m}$ are
two representations, we have that \cite[Theorem~4.6]{zbMATH04007261}
\begin{equation}\label{eq:wsf.ccs}
 \widehat{c}_{\rho\oplus\tau}=\widehat{c}_{\rho}*\widehat{c}_{\tau}.
\end{equation}

\begin{remark}\label{rem:uccs}
Let $\widehat{C}_k\in \widehat{H}^{2k}(\BG[\GLn{n}];\CC/\ZZ)$ be the $k$-th universal Chern differential character.
Since the pull-back $(\gamma^n_d,\BG[\GLn{n}]^d,\nabla_\mathrm{univ}^d)$ of the universal bundle with universal connection $(\gamma^n,\BG[\GLn{n}],\nabla_\mathrm{univ})$  by the map $\iota \colon \BG[\GLn{n}]^d \to \BG[\GLn{n}]$ is flat
(see Remark~\ref{rem:universal.connection}), we have that
\begin{equation*}
 \widehat{c}_k=\iota^*(\widehat{C}_k)\in H^{2k-1}(\BG[\GLn{n}]^d;\CC/\ZZ).
\end{equation*}

Since the vector bundle $V_\rho\to M$ has a flat connection, its classifying map with connection $\bar{f}\colon M\to \GLn{n}$ factorizes as in \eqref{eq:flat.fact}, i.~e., $\bar{f}=\iota\circ f$,
with $f\colon M\to\GLn{n}^d$.
Hence, the $k$-th Cheeger-Chern-Simons class of the representation $\rho$ is given by
\begin{equation*}
 \widehat{c}_{\rho,k}=f^*(\widehat{c}_k).
\end{equation*}
The class $\widehat{c}_k\in H^{2k-1}(\BG[\GLn{n}]^d;\CC/\ZZ)$ is called the \emph{universal $k$-th Cheeger-Chern-Simons class for flat bundles}
\end{remark}

\begin{remark}\label{rem:G.M}
The Chern classes $c_k(\rho)$ of a representation $\rho\colon \Gamma=\pi_1(M)\to\GLn{n}$ defined in Subsection~\ref{ssec:cc.rho}
are cohomology classes of the classifying space $B\Gamma$,
while the Chern classes $c_{\rho,k}$ defined in this subsection are cohomology classes of $M$.
They are related in the following way. Since $\pi_1(M)=\Gamma$ there is a map $\phi\colon M\to \BG[\Gamma]$.
Then $c_{\rho,k}=\phi^*(c_k(\rho))$.
\end{remark}

\begin{remark}\label{rem:G.M.ccs}
One can also define the Cheeger-Chern-Simons classes of a representation $\rho\colon \Gamma\to\GLn{n}$ as cohomology classes of the classifying space $\BG[\Gamma]$.
Let $\widehat{c}_k\in H^{2k-1}(\BG[\GLn{n}]^d;\CC/\ZZ)$ be the $k$-th universal Cheeger-Chern-Simons class (see Remark~\ref{rem:uccs}).
Let $\BG[\rho] \colon \BG[\Gamma] \to \BG[\GLn{n}]^d$ be the map induced between classifying spaces by $\rho$.
The \emph{$k$-th Cheeger-Chern-Simons class of $\rho$} is the pullback
\begin{equation}\label{eq:ccs.BG}
 \widehat{c}_k(\rho)=\BG[\rho]^*(\widehat{c}_k)\in H^{2k-1}(\BG[\Gamma];\CC/\ZZ).
\end{equation}
Let $M$ be a smooth manifold with fundamental group $\pi_1(M)=\Gamma$ and consider the map $\phi\colon M\to \BG[\Gamma]$.
Then $\widehat{c}_{\rho,k}=\phi^*(\widehat{c}_k(\rho))$.
\end{remark}

Let $\Gamma$ be a discrete group. Let $\rho \colon \Gamma \to \GLn{n}$ be a representation.
By \eqref{eq:ccs.BG} we have that
$\widehat{c}_k(\rho)\in H^{2k-1}(\BG[\Gamma];\CC/\ZZ)$. By the Universal Coefficient Theorem and using that $\CC/\ZZ$ is divisible (equivalently, injective~\cite[Corollary~4.2]{zbMATH03935317}) we have
\begin{equation}\label{eq:Iso.Hom.Chern}
    H^{2k-1}(\BG[\Gamma];\CC/\ZZ) \cong \Hom(H_{2k-1}(\BG[\Gamma];\ZZ),\CC/\ZZ).
\end{equation}
Thus, the Chern-Cheeger-Simons classes $\widehat{c}_k(\rho)$ of $\rho$ can be identified as morphisms
\begin{equation*}
 \widehat{c}_k(\rho) \colon H_{2k-1}(\BG[\Gamma];\ZZ) \to \CC/\ZZ,
\end{equation*}
\begin{definition}\label{def:CCS.number}
Let $\kappa$ be a homology class in $H_{2k-1}(\BG[\Gamma];\ZZ)$, we call the image
\begin{equation*}
\widehat{c}_k(\rho)(\kappa)\in\CC/\ZZ,
\end{equation*}
the \emph{$k$-th Cheeger-Chern-Simons-number (CCS-number) of $\rho$} with respect to $\kappa$.
\end{definition}

A representation $\rho \colon \pi_1(M) \to \GLn{n}$ is said to be \emph{topologically trivial} if the vector bundle $V_\rho\to M$ is topologically trivial, i.\ e., it is isomorphic
as a topological vector bundle to the product bundle $M\times\CC^n\to M$.

\section{CCS-numbers of compact oriented 3-manifolds}

In this section, we restrict to the case when the manifold is a compact oriented $3$-manifold $L$.
We define the first and second CCS-numbers of a representation $\rho\colon\Gamma=\pi_1(L)\to \GLn{n}$ and
we prove that the second CCS-number for topologically trivial representations $\rho$ agrees with the topological index of the Dirac operator on $L$ with respect to the flat complex vector bundle associated to $\rho$, defined by Atiyah, Patodi and Singer in  \cite[p.~87]{Atiyah-Patodi-Singer:SARGIII}.
For the case when $L$ is a rational homology $3$-sphere, using the Index Theorem for flat bundles, we give a formula for the second CCS-number of any representation.

\subsection{CCS-numbers of representations}

Let $L$ be a compact oriented $3$-manifold and consider a representation $\rho\colon\Gamma=\pi_1(L)\to \GLn{n}$. By \eqref{eq:CCS.rep} we have that
$\widehat{c}_{\rho,k}\in H^{2k-1}(M;\CC/\ZZ)$. By the same argument of \eqref{eq:Iso.Hom.Chern} there is an isomorphism $H^{2k-1}(L;\CC/\ZZ) \cong \Hom(H_{2k-1}(L;\ZZ),\CC/\ZZ)$.
\begin{definition}
Let $\nu$ be a homology class in $H_{1}(L;\ZZ)$, we call the image
\begin{equation*}
\widehat{c}_{\rho,1}(\nu)\in\CC/\ZZ,
\end{equation*}
the \emph{first CCS-number of $\rho$} with respect to $\nu$.
Let $[L]\in H_{3}(L;\ZZ)$ be the fundamental class. We define the \emph{second CCS-number of $\rho$} by
\begin{equation*}
\widehat{c}_{\rho,2}([L])\in\CC/\ZZ.
\end{equation*}
\end{definition}
Notice that the second CCS-number of $\rho$ with respect to another class in $H_{3}(L;\ZZ)$ is a multiple of $\widehat{c}_{\rho,2}([L])$.
The relation with Definition~\ref{def:CCS.number} is the following. Since $\pi_1(L)=\Gamma$ there is a map $\phi\colon L\to \BG[\Gamma]$.
Let $\nu\in H_{2k-1}(L;\ZZ)$, with $k=1,2$. Then we have
\begin{equation}\label{eq:ccsL.ccsBG}
\widehat{c}_{\rho,k}(\nu)=\phi^*(\widehat{c}_{k}(\alpha))(\nu)=\widehat{c}_{k}(\rho)(\phi_*(\nu)).
\end{equation}

\subsection{The index theorem for flat bundles}

Let $A$ be a self-adjoint elliptic operator acting on the space of sections $C^\infty(M,E)$ of a vector bundle $E\to M$ over a compact manifold $M$.
Then $A$ has a discrete spectrum with real eigenvalues $\{\lambda\}$. In \cite{Atiyah-Patodi-Singer:SARGI} Atiyah, Patodi and Singer define a complex
valued function $\eta(s;A)$, of the complex variable $s$, called the
\emph{$\eta$-series} of $A$ by
\begin{equation*}
\eta(s;A)=\sum_{\lambda\not=0}(\sign\lambda)\abs{\lambda}^{-s},
\end{equation*}
where the sum is taken over the non-zero eigenvalues of $A$. This
series converges when the real part $\Re(s)$ of $s$ is sufficiently large and by results of
Seeley \cite{Seeley:CompPowers} extends by analytic continuation
to a meromorphic function on the whole $s$-plane and it is finite at
$s=0$.
The number $\eta(A)=\eta(0;A)$ is called the \emph{$\eta$-invariant} of $A$
and it is a spectral invariant which measures the asymmetry of the
spectrum of $A$.
They also define a refinement of the $\eta$-series which takes into
account the zero eigenvalues of $A$ by setting $\xi(s;A)=\frac{h+\eta(s;A)}{2}$, where $h$ is the dimension of the
kernel of $A$.

Consider a unitary representation $\rho\colon\pi_1(M)\rightarrow U(n)$ and the flat vector bundle $V_\rho$ given by $\rho$.
One can couple the operator $A$ to $V_{\rho}$ to get an operator $A_{\rho}$ acting on $C^\infty(M,E\otimes V_{\rho})$ which is again self-adjoint.
The reduced $\eta$ and $\xi$-series are given by
\begin{equation}\label{eq:reduced}
\tilde{\eta}_\rho(s;A)=\eta(s;A_\rho)-n\eta(s;A),\quad \tilde{\xi}_\rho(s;A)=\xi(s;A_\rho)-n\xi(s;A),
\end{equation}
where $n$ is the dimension of the representation $\rho$.
Following \cite[Section~2]{Atiyah-Patodi-Singer:SARGIII} the functions \eqref{eq:reduced} are finite at $s=0$ and reducing
modulo $\ZZ$ then
\begin{equation*}
\tilde{\eta}_\rho(A)=\tilde{\eta}_\rho(0;A)\in\RR/\ZZ,\qquad
\tilde{\xi}_\rho(A)=\tilde{\xi}_\rho(0;A)\in\RR/\ZZ,
\end{equation*}
are homotopy invariants of $A$.

If the representation $\rho\colon\pi_1(M)\to \GLn{n}$ is not unitary the twisted operador $A_\rho$ is not self-adjoint, but it has a self-adjoint symbol and this is
enough to define $\eta(s;A_\rho)$. To define $\xi(s;A_\rho)$ one takes $h$ to be the dimension arising from the spectrum on the imaginary axis and now $\tilde{\xi}_\rho(0;A)$
is a complex number modulo integers, that is $\tilde{\xi}_\rho(A)\in\CC/\ZZ$.

If $A$ is a self-adjoint elliptic operator of order $m$ acting on the space of sections $C^\infty(M,E)$ of a vector bundle $E\to M$ over $M$,
its homotopy class depends only on the homotopy class of its leading symbol $\sigma_m(A)$.
Denote by $T^*M$ \emph{the cotangent bundle} of $M$. By \cite[Proposition~3.1]{Atiyah-Patodi-Singer:SARGIII} there is a one-to-one correspondence between the stable classes of self-adjoint symbols on $M$ and the elements of $K^1(T^*M)$.

Fixing a representation $\rho\colon\pi_1(M)\to \GLn{n}$, the map $A\mapsto \tilde{\xi}_\rho(A)$ induces a homomorphism \cite[p.~87]{Atiyah-Patodi-Singer:SARGIII}
\begin{align*}
    \ind_\rho \colon K^1(T^*M) &\to \CC / \ZZ\\
    \sigma_m(A) &\mapsto \tilde{\xi_\rho}(A)
\end{align*}
called \emph{the analytical index of $A$ (or of the symbol class of $A$) with respect to the flat bundle given by $\rho$}.

On the other hand, the representation $\rho\colon\pi_1(M)\to \GLn{n}$ defines an element $[\rho]\in K^{-1}(M,\CC/\ZZ)$ in K-theory with coefficients in $\CC/\ZZ$,
see \cite[p.~90]{Atiyah-Patodi-Singer:SARGIII}. There is a pairing
\begin{equation*}
 K^{-1}(M,\CC/\ZZ)\otimes K^1(T^*M)\to K^0(T^*M,\CC/\ZZ),
\end{equation*}
and a homomorphism
\begin{equation*}
 \Ind\colon K^0(T^*M,\CC/\ZZ)\to\CC/\ZZ,
\end{equation*}
which extends the usual topological index of \cite{Atiyah-Singer:IEOI}. For each representation $\rho$ there is a topologically defined homomorphism
\begin{align*}
    \Ind_\rho \colon K^1(T^*M) &\to \CC/\ZZ\\
    \sigma_m(A) &\mapsto -\Ind([\rho]\cdot \sigma_m(A))
\end{align*}
called the \emph{the topological index with respect to the flat bundle given by $\rho$}.

The \emph{Index Theorem for flat bundles} \cite[(5.3)]{Atiyah-Patodi-Singer:SARGIII}) states that if $M$ is odd dimensional, then $\ind_\rho$ coincides with $\Ind_\rho$.

\begin{remark}
If $\rho\colon\pi_1(M)\rightarrow U(n)$ is a unitary representation then both indices $\ind_\rho$ and $\Ind_\rho$ have values in $\RR/\ZZ$.
\end{remark}

\subsection{The second CCS-number of \texorpdfstring{$\rho$}{r}}\label{ssec:2ndcn}

Let $L$ be a compact oriented $3$-manifold, then it is a spin-manifold \cite[Chapter~VII, Theorem~1]{zbMATH00042662} and it has a Dirac operator $D$, which is a first order, self-adjoint, elliptic operator, acting on sections of the spinor bundle (see~\cite[\S~3.4]{zbMATH06755669} or \cite[Example~5.9]{Lawson-Michelsohn:SpinGeo}).

\begin{theorem}\label{th:C2.APS}
Let $L$ be a compact oriented $3$-manifold.
Let $D$ be the Dirac operator on $L$ and let $\sigma=[\sigma_1(D)]$ be the symbol class of $D$ in $K^1(T^*L)$.
Let $\rho \colon\pi_1(L) \to \mathrm{GL}(n,\CC)$ be a topologically trivial representation. Then, the second CCS-number of $\rho$ and the topological index of $D$ with respect to
$\rho$, coincide, that is,
\begin{equation*}
 \widehat{c}_{\rho,2}([L])=\Ind_\rho\sigma.
\end{equation*}
\end{theorem}

\begin{proof}
The proof consists in to identify $\Ind_\rho\sigma$ for the particular case when $\sigma$ is the symbol class of the Dirac operator and $\rho$ is a topologically trivial representation
of the fundamental group of a compact oriented $3$-manifold, following the proof of the Index Theorem for flat bundles.
In \cite[pp.~88]{Atiyah-Patodi-Singer:SARGIII} $K^1(L,\CC/\ZZ)$ is defined as the cokernel of the following natural map
\begin{equation*}
    K^{-1}(L,\QQ) \to  K^{-1}(L,\QQ/\ZZ) \oplus K^{-1}(L,\CC).
\end{equation*}
Thus, given a representation $\rho\colon\pi_1(L)\to GL(n,\CC)$, the element $[\rho]\in K^{-1}(L,\CC/\ZZ)$ is given by two elements
\begin{equation*}
    a_\rho \in K^{-1}(L,\QQ/\ZZ) \quad \text{and} \quad b_\rho \in K^{-1}(L,\CC).
\end{equation*}
In terms of these elements, the topological index can be writen as (see~\cite[p.~98]{Atiyah-Patodi-Singer:SARGIII})
\begin{equation}
\label{Th.C2.1}
    \Ind_\rho\sigma= -\Ind_{\CC}(b_\rho \sigma)+\Ind_{\ZZ/k\ZZ}(\sigma a_\rho) \mod \ZZ,
\end{equation}
where $\Ind_{\CC}$ is induced from the topological index $K(T^*L)\to \ZZ$ by tensoring with $\CC$ and $\Ind_{\ZZ/k\ZZ}$ is the topological index for $\ZZ/k\ZZ$-coefficients (see \cite[Proposition~6.2 and (8.4)]{Atiyah-Patodi-Singer:SARGIII}).

Let us recall the construction of the elements $a_\rho$ and $b_\rho$ when
$\rho \colon \pi_1(L) \to GL(n,\CC)$ is a topologically trivial representation and $s\colon V_\rho \cong L\times\CC^n$ is a trivialization.
In this case we have that $a_\rho = 0$ (see \cite[p.~89, (i), (ii)]{Atiyah-Patodi-Singer:SARGIII}).
The element $b_\rho$ is constructed as follows. 
The bundle $V_\rho$ has two different flat connections, its canonical flat connection $\nabla_\rho$ and a connection $\nabla_\text{triv}$
coming from the trivial connection on $L\times\CC^n$ by the isomorphism $s$.
Applying the Chern-Simons construction to the bundle $V_\rho$, using the curve of connections $\nabla_t=t\nabla_\text{triv}+(1-t)\nabla_\rho$ and the invariant formal power series $ch(A)$ given in  \eqref{eq:Chern.power} for the Chern character, we get the Chern-Simons form $Tch(\nabla_\text{triv},\nabla_\rho)\in\Omega^\mathrm{odd}(L)$ as in \eqref{eq:CS.form} and a mixed odd-dimensional
cohomology class $\beta=[Tch(\nabla_\text{triv},\nabla_\rho)]\in\ H^\mathrm{odd}(L;\CC)$. We have that (see \cite[p.~89, (iii)]{Atiyah-Patodi-Singer:SARGIII})
\begin{equation}\label{eq:b.beta}
 b_\rho=ch^{-1}(\beta)
\end{equation}
where $ch\colon K^{-1}(L,\CC)\to H^{\mathrm{odd}}(L;\CC)$ is the Chern character.

By Karoubi~\cite{zbMATH04057753,zbMATH04180348}, there is an exact sequence as follows
\begin{equation}
\label{eq:Karoubi1}
    \dots \to K^{-1}(L) \to K^{-1}(L,\CC) \to K^{-1}(L;\CC/\ZZ) \to K^0(L) \to \dots
\end{equation}
By changing the trivialization of $V_\rho$ we obtain a different element in $K^{-1}(L,\CC)$. However, changing the trivialization is equivalent to take an element in $K^{-1}(L)$. Therefore by the exact sequence~\eqref{eq:Karoubi1}, it is clear that the class of $b_\rho$ in $K^{-1}(L;\CC/\ZZ)$ does not depend on the choice of trivialization (compare with \cite[p.~89, (iv)]{Atiyah-Patodi-Singer:SARGIII}).
Thus, when $\rho$ is topologically trivial \eqref{Th.C2.1} becomes
\begin{equation}\label{eq:RedInd1}
    \Ind_\rho\sigma =-\Ind_{\CC}(b_\rho \sigma) \mod \ZZ.
\end{equation}
Since $\sigma$ is the symbol class of the Dirac operador, \eqref{eq:RedInd1} coincides with the direct image homomorphism $K^{-1}(L,\CC/\ZZ)\to\CC/\ZZ$ (see \cite[p.~82~(2)]{Atiyah-Patodi-Singer:SARGIII}), which is given cohomologically by
\begin{align*}
 \Ind_\rho \sigma&= \bigl\{\mathrm{ch}(b_\rho) \cdot \mathscr{I}(L) \bigr\}[L] \mod \ZZ,&&\\
 &= \bigl\{\beta \cdot \mathscr{I}(L) \bigr\}[L] \mod \ZZ,&&\text{by \eqref{eq:b.beta}}
\end{align*}
where $\mathscr{I}(L)$ is the index class of $L$ (see \cite[p.~556]{Atiyah-Singer:IEOIII}).
Since $L$ is a compact $3$-manifold, it is parallelizable and $\mathscr{I}(L)=1$. Thus,
\begin{align}
 \Ind_\rho \sigma&= \bigl\{\beta\bigr\}[L] \mod \ZZ,\notag\\
 &= \bigl\{[Tch(\nabla_\text{triv},\nabla_\rho)]\bigr\}[L] \mod \ZZ,\label{eq:ind.tch}
\end{align}
We have that $\beta=[Tch(\nabla_\text{triv},\nabla_\rho)]\in\ H^\mathrm{odd}(L;\CC)$, let us denote the component of degree $3$ of the Chern-Simons form $Tch(\nabla_\text{triv},\nabla_\rho)$ by
$Tch(\nabla_\text{triv},\nabla_\rho)_3$, so we have $[Tch(\nabla_\text{triv},\nabla_\rho)_3]\in H^3(L;\CC)$.
Now consider the commutative diagram
\begin{equation*}
 \xymatrix{
H^3(L;\CC)\ar[r]^{p_\ZZ}\ar[rd]^{p_\QQ} & H^3(L;\CC/\ZZ)\ar[d]^{p}\\
 &  H^3(L;\CC/\QQ).
 }
\end{equation*}
By construction we have that
\begin{equation*}
\widehat{c}_1(V_\rho,\nabla_\rho)*\widehat{c}_1(V_\rho,\nabla_\rho),\widehat{c}_2(V_\rho,\nabla_\rho),\widehat{c}_1(V_\rho,\nabla_\text{triv})*\widehat{c}_1(V_\rho,\nabla_\text{triv}),\widehat{c}_2(V_\rho,\nabla_\text{triv})\in H^3(L;\CC/\ZZ).
\end{equation*}
By \eqref{eq:CS.class.char} we have that 
\begin{equation*}
p_\QQ\bigl([Tch(\nabla_\text{triv},\nabla_\rho)_3]\bigr)=\widehat{ch}_2(V_\rho,\nabla_\text{triv})-\widehat{ch}_2(V_\rho,\nabla_\rho)\in H^3(L;\CC/\QQ),
\end{equation*}
since the component of degree $3$ of $\widehat{ch}(V_\rho,\nabla)$ is $\widehat{ch}_2(V_\rho,\nabla)$.
By \eqref{eq:ch2} we can decompose $Tch(\nabla_\text{triv},\nabla_\rho)_3$ as follows
\begin{equation*}
    Tch(\nabla_\text{triv},\nabla_\rho)_3 = \frac{1}{2}T(C_1\cdot C_1)(\nabla_\text{triv},\nabla_\rho)-TC_2(\nabla_\text{triv},\nabla_\rho),
\end{equation*}
where $C_1\cdot C_1$ is the product in the algebra of invariant polynomials $I(\GLn{n})$.

Set $\omega_1=T(C_1\cdot C_1)(\nabla_\text{triv},\nabla_\rho)$ and $\omega_2=TC_2(\nabla_\text{triv},\nabla_\rho)$.  By~\eqref{eq:CS.class.char} we get
\begin{align*}
p_\ZZ ([\omega_1])&=\widehat{c}_1(V_\rho,\nabla_\text{triv})*\widehat{c}_1(V_\rho,\nabla_\text{triv}) - \widehat{c}_1(V_\rho,\nabla_\rho)*\widehat{c}_1(V_\rho,\nabla_\rho),\\
p_\ZZ([\omega_2])&=\widehat{c}_2(V_\rho,\nabla_\text{triv})-\widehat{c}_2(V_\rho,\nabla_\rho),\\
p_\QQ \Bigl(\frac{1}{2}[\omega_1]\Bigr)&=\frac{1}{2}p\bigl(\widehat{c}_1(V_\rho,\nabla_\text{triv})*\widehat{c}_1(V_\rho,\nabla_\text{triv}) - \widehat{c}_1(V_\rho,\nabla_\rho)*\widehat{c}_1(V_\rho,\nabla_\rho) \bigr),\\
p_\QQ([\omega_2])&=p\bigl(\widehat{c}_2(V_\rho,\nabla_\text{triv})\bigr)-p\bigl(\widehat{c}_2(V_\rho,\nabla_\rho)\bigr).
\end{align*}
In order to finish the proof we only need to prove that $[\omega_1]=0$. Indeed, if we assume that $[\omega_1]=0$, then
\begin{equation*}
p_\ZZ\bigl([Tch(\nabla_\text{triv},\nabla_\rho)_3]\bigr)=p_\ZZ\left(\frac{1}{2}[\omega_1]-[\omega_2]\right)=p_\ZZ\bigl(-[\omega_2]\bigr)=-\widehat{c}_2(V_\rho,\nabla_\text{triv})+\widehat{c}_2(V_\rho,\nabla_\rho).
\end{equation*}
By \cite[Proposition~2.10]{zbMATH04007261} we have that $\widehat{c}_2(V_\rho,\nabla_\text{triv})=0$. Hence,
\begin{equation*}
 p_\ZZ\bigl([Tch(\nabla_\text{triv},\nabla_\rho)_3]\bigr)=\widehat{c}_2(V_\rho,\nabla_\rho)\in H^3(L;\CC/\ZZ),
\end{equation*}
and together with \eqref{eq:ind.tch} we get
\begin{equation*}
 \Ind_\rho \sigma=p_\ZZ\bigl([Tch(\nabla_\text{triv},\nabla_\rho)_3]\bigr)[L]=\widehat{c}_2(V_\rho,\nabla_\rho)[L]=\widehat{c}_{\rho,2}[L]\in\CC/\ZZ.
\end{equation*}
Thus, we only have to prove that  $[\omega_1]=0$. 
Let $\widetilde{\nabla}$ be the connection on $V_\rho\times I$ defined by the curve of connections $\nabla_t$ and let $\widetilde{\Theta}$ be its curvature. 
By~\eqref{eq:CS.form} we get,
\begin{equation*}
\omega_1=\pi_*\left ((C_1\cdot C_1)(\widetilde{\Theta})\right)=\pi_*\left (C_1(\widetilde{\Theta})\wedge C_1(\widetilde{\Theta})\right)\in\Omega^{3}(L),
\end{equation*}
where $\pi_*$ is integration along the fibres \cite[Remark~19-1.8]{Husemoller:Bundles}. 
Note that $[C_1(\widetilde{\Theta})]\in H^2(L\times I,\CC)$.
Since $\rho \colon\pi_1(L) \to \mathrm{GL}(n,\CC)$ is topologically trivial, $V_\rho\to M$ is a trivial vector bundle and $V_\rho\times I\to M\times I$ is also a trivial
vector bundle. Hence $[C_1(\widetilde{\Theta})]=0$.
Therefore, there exists $\mathcal{C} \in \Omega^1(L\times I)$ such that $d \mathcal{C}=C_1(\widetilde{\Theta})$. Using the form $\mathcal{C}$ we get
\begin{equation*}
    C_1(\widetilde{\Theta})\wedge C_1(\widetilde{\Theta})=C_1(\widetilde{\Theta})\wedge d \mathcal{C}=-d(C_1(\widetilde{\Theta})\wedge \mathcal{C}).
\end{equation*}
By the Homotopy Formula for fiber integration (see~\cite[Proposition~19-1.9]{Husemoller:Bundles}), we get
\begin{equation*}
-\pi_*\bigl ( d(C_1(\widetilde{\Theta})\wedge \mathcal{C})  \bigr) =  d \pi_* \bigl ( C_1(\widetilde{\Theta})\wedge \mathcal{C} \bigr)-(j_1-j_0)\bigl(C_1(\widetilde{\Theta})\wedge \mathcal{C}\bigr).
\end{equation*}
We claim that the last term $(j_1-j_0)\bigl(C_1(\widetilde{\Theta})\wedge \mathcal{C}\bigr)=0$.
Assuming this, we have
\begin{equation*}
\omega_1=\pi_*\left (C_1(\widetilde{\Theta})\wedge C_1(\widetilde{\Theta})\right)=-\pi_*\left ( d(C_1(\widetilde{\Theta})\wedge \mathcal{C})  \right) =  d \pi_* \left ( C_1(\widetilde{\Theta})\wedge\mathcal{C} \right).
\end{equation*}
Therefore, the form $\omega_1$ is exact, $[\omega_1]=0 \in H^3(L,\CC)$ and we reach the conclusion of the Theorem.

To prove that $(j_1-j_0)\bigl(C_1(\widetilde{\Theta})\wedge \mathcal{C}\bigr)=0$ we proceed as follows.
Every differential form $\omega\in\Omega^k(L\times I)$ can be decomposed as $\omega=\alpha+\beta\wedge dt$ where $\alpha\in\Omega^k(L\times I)$, $\beta\in\Omega^{k-1}(L\times I)$ each without
any $dt$ factor \cite[Remark~19-1.8]{Husemoller:Bundles}. By definition we have that $j_s\colon\Omega^{2k}(L\times I)\to\Omega^{2k-1}(L)$ is given by $j_s(\alpha+\beta\wedge dt)=\alpha|_{t=s}$
with $s=0,1$ \cite[Proposition~19-1.9]{Husemoller:Bundles}.
Let $\theta_\text{triv}$ and $\theta$ be, respectively, the connection forms of the connections $\nabla_\text{triv}$ and $\nabla_\rho$ on the vector bundle $V_\rho$. Since $\nabla_\text{triv}$
is the trivial connection $\theta_\text{triv}=0$. The curvature $\widetilde{\Theta}$ of the connection $\widetilde{\nabla}$ given by the curve of connections $\nabla_t=t\nabla_\text{triv}+(1-t)\nabla_\rho$ is given by (see \cite[Remark~19-5.1]{Husemoller:Bundles})
\begin{equation}\label{eq:curv.conn}
\widetilde{\Theta}=t(1-t)\theta^2+\theta dt.
\end{equation}
Since the form $\mathcal{C}\in\Omega^1(L\times I)$ we can decompose it as $\mathcal{C}=\alpha'+\beta'\wedge dt$, where $\alpha'\in\Omega^1(L\times I)$,
$\beta'\in\Omega^{0}(L\times I)=C^{\infty}(L\times I)$ each without any $dt$ factor.
The invariant polynomial $C_1\in I^1(\GLn{n})$ is given by $C_1(A)=\trace(A)$, thus by \eqref{eq:curv.conn} we have
\begin{equation*}
C_1(\widetilde{\Theta})=t(1-t)\trace(\theta^2)+\trace(\theta)\wedge dt.
\end{equation*}
Therefore
\begin{equation}\label{eq:c1.C}
C_1(\widetilde{\Theta})\wedge \mathcal{C}=t(1-t)\trace(\theta^2)\wedge\alpha'+\bigl(t(1-t)\beta'\trace(\theta^2)-\trace(\theta)\wedge\alpha'\bigr)\wedge dt.
\end{equation}
Setting in \eqref{eq:c1.C}
\begin{equation}\label{eq:alpha}
\alpha=t(1-t)\trace(\theta^2)\wedge\alpha'\in \Omega^3(L\times I),
\end{equation}
and
\begin{equation*}
\beta=\bigl(t(1-t)\beta'\trace(\theta^2)-\trace(\theta)\wedge\alpha'\bigr)\in\Omega^{2}(L\times I)
\end{equation*}
we have that both of them do not have any $dt$ factor and $C_1(\widetilde{\Theta})\wedge \mathcal{C}=\alpha+\beta\wedge dt$.
Therefore, evaluating $t=1$ and $t=0$ in \eqref{eq:alpha} we get
\begin{equation*}
(j_1-j_0)\bigl(C_1(\widetilde{\Theta})\wedge \mathcal{C}\bigr)=\alpha|_{t=1}-\alpha|_{t=0}=0.
\end{equation*}
This finishes the proof.
\end{proof}

By the Index Theorem for flat bundles we have the following corollary.
\begin{corollary}\label{cor:c2.xi}
With the hypothesis of Theorem~\ref{th:C2.APS} we have
 $\widehat{c}_{\rho,2}([L])=\tilde{\xi}_\rho(D)$.
\end{corollary}

\begin{remark}\label{rem:gen.JW}
If $L$ is an integral homology sphere, then Theorem~\ref{th:C2.APS} is equivalent to the result of Jones and Westbury~\cite[Theorem~A]{Jones-Westbury:AKTHSEI}, see Section~\ref{sec:k3}.
\end{remark}

\subsection{CCS-numbers of rational homology 3-spheres}\label{ssec:cn.rhs}

Let $L$ be a rational homology $3$-sphere.
Let $\rho\colon\pi_1(L)\to \GLn{n}$ be a representation.
By Poincaré duality we have that $H^1(L;\CC)=H^2(L;\CC)=0$, by the cohomology long exact sequence \eqref{eq:coef.CheegerSimons2} of $L$ corresponding to the short exact sequence \eqref{eq:coef.CheegerSimons1} we have a correspondence between the first \ccs{} class and the first Chern class of $\rho$ under the isomorphism
\begin{equation}\label{eq:ccs-c1}
\begin{aligned}
H^1(L;\CC/\ZZ) &\cong H^2(L;\ZZ)\\
\widehat{c}_{\rho,1} &\mapsto c_{\rho,1}.
\end{aligned}
\end{equation}

\begin{lemma}\label{lem:ci=0}
Let $L$ be a rational homology $3$-sphere and $\rho\colon\pi_1(L)\to \GLn{n}$ a representation of its fundamental group.
The following are equivalent:
\begin{enumerate}[1.]
 \item The representation $\rho$ is topologically trivial.\label{it:rho.tt}
 \item $c_{\rho,1}=0$.\label{it:c1=0}
 \item $\widehat{c}_{\rho,1}=0$.\label{it.hat.c1=0}
\end{enumerate}
\end{lemma}

\begin{proof}
The equivalence of \ref{it:c1=0} and \ref{it.hat.c1=0} is given by the isomorphism \eqref{eq:ccs-c1}.
The equivalence of \ref{it:rho.tt} and \ref{it:c1=0} is given as follows.
For $n=1$, $c_{\rho,1}=0$ if and only if $\rho$ is topologically trivial since line bundles are classified by its first Chern class.
If $n\geq2$, it follows from \cite[Theorem~3.2]{zbMATH03199402}  since the only possible non-zero Chern class is $c_{\rho,1}$.
\end{proof}

\begin{lemma}\label{lem:topo.trivialization}
Let $L$ be a rational homology $3$-sphere and $\rho\colon\pi_1(L)\to \GLn{n}$ a representation of its fundamental group.
Then the representation $\rho\oplus\det(\rho^*)$, where $\rho^*$ is the dual representation, is topologically trivial.
\end{lemma}

\begin{proof}
By the Whitney product formula, Remark~\ref{rem:C1Det} and the fact that
$c_{\rho^*,1}=-c_{\rho,1}$, we have that $c_{\rho\oplus \det(\rho^*),1}=c_{\rho,1}+c_{\det(\rho^*),1}=c_{\rho,1}+c_{\rho^*,1}=0$. Thus, by Lemma~\ref{lem:ci=0} the representation
$\rho\oplus \det(\rho^*)$ is topologically trivial.
\end{proof}

\begin{lemma}\label{lem:hc1.det}
Let $L$ be a rational homology $3$-sphere. For any representation $\rho\colon\pi_1(L)\to \GLn{n}$, the first Cheeger-Chern-Simons class satisfies
\begin{equation*}
    \widehat{c}_{\rho,1}=\widehat{c}_{\det(\rho),1}.
\end{equation*}
\end{lemma}

\begin{proof}

By Remark~\ref{rem:C1Det} and Remark~\ref{rem:G.M} $c_{\rho,1}=c_{\det(\rho),1}$ and both $\widehat{c}_{\rho,1}$ and $\widehat{c}_{\det(\rho),1}$ are liftings of $c_{\rho,1}=c_{\det(\rho),1}$ by the
homomorphism $q\colon H^{1}(L;\CC/\ZZ)\to H^{2}(L;\ZZ)$ in \eqref{eq:coef.CheegerSimons2}, but by \eqref{eq:ccs-c1} there is only one such lifting, so they coincide.
\end{proof}

If $L$ is a rational homology $3$-sphere, using Theorem~\ref{th:C2.APS} and the Index Theorem for flat bundles we can give a formula for the second CCS-number of an arbitrary representation $\rho \colon\pi_1(L) \to \mathrm{GL}(n,\CC)$.

\begin{theorem}
\label{Cor:NoTop}
Let $L$ be a rational homology $3$-sphere. Let $\rho \colon \Gamma=\pi_1(L) \to \mathrm{GL}(n,\CC)$ be a representation. Then,
\begin{equation*}
\widehat{c}_{\rho,2}([L])=\tilde\xi_{\rho}(D)-\tilde\xi_{\det(\rho)}(D).
\end{equation*}
\end{theorem}

\begin{proof}
Consider the representation $\rho \oplus \det(\rho^*)$, by Lemma~\ref{lem:topo.trivialization} it is topologically trivial.
Since $c_{\det(\rho^*),1}=-c_{\det(\rho),1}$, by \eqref{eq:ccs-c1} we have that $\widehat{c}_{\det(\rho^*),1}=-\widehat{c}_{\det(\rho),1}$.
Applying \eqref{eq:wsf.ccs} to $\det(\rho)\oplus\det(\rho^*)$ we get
\begin{align*}
\widehat{c}_{\det(\rho)\oplus\det(\rho^*),1}&=\widehat{c}_{\det(\rho),1}+\widehat{c}_{\det(\rho^*),1}=0\\
\widehat{c}_{\det(\rho)\oplus\det(\rho^*),2}&=\widehat{c}_{\det(\rho),1}*\widehat{c}_{\det(\rho^*),1}.
\end{align*}
Hence, by Lemma~\ref{lem:ci=0} $\det(\rho)\oplus\det(\rho^*)$ is topologically trivial.
Applying Corollary~\ref{cor:c2.xi} we have:
\begin{align*}
 \tilde\xi_{\rho}(D)+\tilde\xi_{\det(\rho^*)}(D)=\tilde\xi_{\rho\oplus\det(\rho^*)}(D)&=\widehat{c}_{\rho\oplus\det(\rho^*),2}([L])\\
 &=\widehat{c}_{\rho,1}*\widehat{c}_{\det(\rho^*),1}([L])+\widehat{c}_{\rho,2}([L])\\
 &=\widehat{c}_{\det(\rho),1}*\widehat{c}_{\det(\rho^*),1}([L])+\widehat{c}_{\rho,2}([L])\\
 &=\widehat{c}_{\det(\rho)\oplus\det(\rho^*),2}([L])+\widehat{c}_{\rho,2}([L])\\
 &=\tilde\xi_{\det(\rho)\oplus\det(\rho^*)}(D)+\widehat{c}_{\rho,2}([L])\\
 &=\tilde\xi_{\det(\rho)}(D)+\tilde\xi_{\det(\rho^*)}(D)+\widehat{c}_{\rho,2}([L]).
\end{align*}
From this the theorem follows.
\end{proof}

\section{Elements in \texorpdfstring{$K_3(\CC)$}{K3(C)} via topologically trivial representations}\label{sec:k3}

In this section, given a topologically trivial representation of the fundamental group of a compact oriented $3$-manifold we construct an element in $K_3(\CC)$, the
$3$rd algebraic K-theory group of the complex numbers. With this construction we prove that Theorem~\ref{th:C2.APS} is a generalization of a result by Jones and Westbury \cite[Theorem~A]{Jones-Westbury:AKTHSEI}.
Firstly we have the following lemma.
\begin{lemma}
\label{lema:TopoTrivil}
Let $L$ be a compact oriented $3$-manifold. A representation $\rho\colon\pi_1(L) \to \GLn{n}$ factors through $\SLn{n}$ if and only if it is topologically trivial.
\end{lemma}
\begin{proof}
Notice that $\rho$ factors through $\SLn{n}$ if and only if $\det(\rho)\colon \Gamma\to \GLn{1}$ is the trivial representation.
By \eqref{eq:Vect.Class} the trivial representation corresponds to the trivial (flat) bundle, thus, $c_{\det(\rho),1}=0$.
By Remark~\ref{rem:C1Det} $c_{\rho,1}=c_{\det(\rho),1}=0$ and by \cite[Theorem~3.2]{zbMATH03199402}
this is equivalent to $\rho$ being topologically trivial.
\end{proof}

Let $L$ be a compact oriented $3$-manifold.
Let $\rho\colon\pi_1(L) \to \GLn{n}$ be a topologically trivial representation.
By Lemma~\ref{lema:TopoTrivil} $\rho$ has image in $\SLn{n}^d$, that is, we have $\rho\colon\pi_1(L) \to \SLn{n}^d$.
Let $V_\rho\to L$ be its associated flat vector bundle.
Its classifying map with connection $\bar{f}_n\colon L\to \BG[\GLn{n}]^d$ factorizes as
\begin{equation*}
\xymatrix{
&\BG[\SLn{n}]^d \ar[d]^\iota \\
L \ar[r]^>>>{\bar{f}_n} \ar@{.>}[ru]^{f_n} & \BG[\GLn{n}]^d.
}
\end{equation*}

Let $\widehat{c}_2\in H^3(\BG[\GLn{n}]^d;\CC/\ZZ)$ be the second universal Cheeger-Chern-Simons class (see Remark~\ref{rem:uccs}).
Denote by $\widehat{\widehat{c}}_2$ the Cohomology class
\begin{equation*}
 \widehat{\widehat{c}}_2=\iota^*(\widehat{c}_2)\in H^3(\BG[\SLn{n}]^d;\CC/\ZZ).
\end{equation*}
Then the second Cheeger-Chern-Simons class of $\rho$ is given by
\begin{equation*}
 \widehat{c}_{\rho,2}=f^*_n(\widehat{\widehat{c}}_2)\in H^3(L;\CC/\ZZ).
\end{equation*}
Hence the second CCS-number of $\rho$ is given by
\begin{equation}\label{eq:c2u.ffc}
\widehat{c}_{\rho,2}([L])=f_n^*(\widehat{\widehat{c}}_2)([L])=\widehat{\widehat{c}}_2((f_n)_*([L])),
\end{equation}
that is, evaluating the class $\widehat{\widehat{c}}_2$ on the homology class $(f_n)_*([L])\in H_3(\BG[\SLn{n}]^d;\ZZ)$.

The group $\SLn{n}^d$ is perfect and it is the commutator subgroup of $\GLn{n}^d$. 
Consider the limits $\mathrm{SL}(\CC)= \lim_n \SLn{n}^d$ and  $\mathrm{GL}(\CC)= \lim_n \GLn{n}^d$. 
Hence $\mathrm{SL}(\CC)$ is perfect and it is the commutator subgroup of $\mathrm{GL}(\CC)$.

Let $Bi\colon\BG[\mathrm{SL}(\CC)]\to\BG[\mathrm{GL}(\CC)]$ be the map between classifying spaces induced by the inclusion homomorphism $i\colon\mathrm{SL}(\CC)\to\mathrm{GL}(\CC)$.
Quillen's plus construction is funtorial \cite[Proposition~5.2.4]{Rosenberg:AlgKTheory} and applying it to $Bi$ with respect to $\mathrm{SL}(\CC)$ we get the universal cover of $\BG[\mathrm{GL}(\CC)]^+$ \cite[Thorem~5.2.7]{Rosenberg:AlgKTheory}
\begin{equation*}
 \pi=(Bi)^+\colon\BG[\mathrm{SL}(\CC)]^+\to\BG[\mathrm{GL}(\CC)]^+.
\end{equation*}
Hence, the third algebraic K-Theory group of $\CC$ is given by
\begin{equation*}
K_3(\CC)=\pi_3(\BG[\mathrm{GL}(\CC)]^+)=\pi_3(\BG[\mathrm{SL}(\CC)]^+).
\end{equation*}
By \cite[Proposition~2.5]{Sah:HCLGMDIII} the Hurewicz homomorphism
\begin{equation}\label{eq:hurewicz.SL}
h\colon K_3(\CC)=\pi_3(\BG[\mathrm{SL}(\CC)]^+) \to H_3(\BG[\mathrm{SL}(\CC)]^+;\ZZ)=H_3(\mathrm{SL}(\CC);\ZZ)
\end{equation}
is an isomorphism. On the other hand, we have the following stability isomorphism for $n\geq 3$ \cite[(2.8) and Theorem~4.1~(b)]{Sah:HCLGMDIII}
\begin{equation*}
 H_3(\SLn{n}^d;\ZZ)\to H_3(\mathrm{SL}(\CC);\ZZ).
\end{equation*}
Hence we have that
\begin{equation*}
 K_3(\CC)=H_3(\SLn{n}^d;\ZZ),\qquad n\geq3.
\end{equation*}

Let $L$ be a compact oriented $3$-manifold. Given a topologically trivial representation $\rho\colon\pi_1(L) \to \GLn{n}$ one can construct an element in $K_3(\CC)$ in the following way.
By Lemma~\ref{lema:TopoTrivil} we have $\rho\colon\pi_1(L) \to \SLn{n}^d$. Let $f_n\colon L \to\BG[\SLn{n}]^d$ be the map which induces $\rho$ on the fundamental group and compose it with the natural map $\iota_n\colon \BG[\SLn{n}^d] \to \BG[\mathrm{SL}(\CC)]$ to get
\begin{equation*}
f=\iota_n\circ f_n\colon L\to \BG[\mathrm{SL}(\CC)].
\end{equation*}
Consider the homomorphism induced in homology
\begin{equation*}
f_*\colon H_3(L;\ZZ)\to H_3(\BG[\mathrm{SL}(\CC)];\ZZ)\cong K_3(\CC).
\end{equation*}
The image of the fundamental class $[L]$ gives the element $\langle L,\rho\rangle=f_*([L])\in K_3(\CC)$.

In the case when $L$ is a rational homology sphere, we can extend the construction for representations $\rho\colon\pi_1(L) \to \GLn{n}$ that are not topologically trivial taking the representation
$\rho\oplus \det(\rho^*)\colon\pi_1(L) \to \SLn{n+1}$ which by Lemma~\ref{lem:topo.trivialization} is topologically trivial. Then take $\langle L,\rho\oplus \det(\rho^*)\rangle\in K_3(\CC)$.

Let $\Sigma$ be an integral homology $3$-sphere and $\rho\colon\pi_1(\Sigma)\to \GLn{n}$ a representation. 
Since $\pi_1(\Sigma)$ is perfect, every complex representation must have image in $\SLn{n}$, that is, the representation $\rho$ is topologically trivial
Lemma~\ref{lema:TopoTrivil}.
In \cite{Jones-Westbury:AKTHSEI} Jones and Westbury constructed an element $[\Sigma,\rho]$ in $K_3(\CC)$ as follows.
Consider again the map $f=\iota_n\circ f_n\colon \Sigma\to \BG[\mathrm{SL}(\CC)]$.

Applying Quillen plus construction to $f$ we obtain a map
\begin{equation*}
f^+ \colon \Sigma^+\cong \mathbb{S}^3 \to \BG[\mathrm{SL}(\CC)]^+.
\end{equation*}
Its homotopy class $[f^+]$ is an element $[\Sigma,\rho]=[f^+]\in\pi_3(\BG[\mathrm{SL}(\CC)]^+)=K_3(\CC)$.

In the case of an integral homology sphere $\Sigma$ one can compare the two constructions and see that they give the same element.

The maps $f$ and $f^+$ fit in the following commutative diagram
\begin{equation}\label{eq:f+}
    \xymatrix{
    \Sigma \ar[r]^{f} \ar[d]^{+} & \BG[\mathrm{SL}(\CC)]\ar[d]^{+} \\
    \mathbb{S}^3 \ar[r]^{f^+} & \BG[\mathrm{SL}(\CC)]^+.
    }
\end{equation}
Diagram \eqref{eq:f+} induces the following diagram in homology
\begin{equation}\label{eq:H.f+}
    \xymatrix{
    H_3(\Sigma;\ZZ) \ar[r]^{f_*} \ar[d]^{\cong} & H_3(\BG[\mathrm{SL}(\CC)];\ZZ) \ar[d]^{\cong} \\
    H_3(\mathbb{S}^3;\ZZ) \ar[r]^{f^+_*} & H_3(\BG[\mathrm{SL}(\CC)]^+;\ZZ)
    }
\end{equation}
Therefore,
\begin{equation}\label{eq:K3=}
\langle\Sigma,\rho\rangle=\tilde{f}_*([\Sigma])=\tilde{f}^+_*([S^3])=h([\tilde{f}^+])=h([\Sigma,\rho]),
\end{equation}
where $h$ is the Hurewicz isomorphism \eqref{eq:hurewicz.SL}.
Thus, we have proved the following

\begin{theorem}
Let $\Sigma$ be a integral homology $3$-sphere and $\rho\colon\pi_1(\Sigma)\to \SLn{n}$ a representation. Then $\langle\Sigma,\rho\rangle=[\Sigma,\rho]\in K_3(\CC)$.
\end{theorem}

In \cite[Theorem~A]{Jones-Westbury:AKTHSEI} the authors consider a homomorphism $e\colon K_3(\CC)\to\CC/\ZZ$ and prove that given a representation $\rho\colon\pi_1(\Sigma)\to \SLn{n}$
of the fundamental group of a integral homology $3$-sphere $\Sigma$ one has $e([\Sigma,\rho])=\tilde{\xi}_\rho(D)$, where $D$ is the Dirac operator on $\Sigma$.

On the other hand, expanding diagram \eqref{eq:H.f+} we get
\begin{equation}\label{eq:H.f+.e}
    \xymatrix{
    H_3(\Sigma;\ZZ) \ar[r]^{(f_n)_*} \ar[d]^{\cong} & H_3(\BG[\SLn{n}^d];\ZZ) \ar[d]^{\cong}\ar[r]^{(\iota_n)_*}_{\cong} & H_3(\BG[\mathrm{SL}(\CC)];\ZZ) \ar[d]^{\cong}\\
    H_3(\mathbb{S}^3;\ZZ) \ar[r]^{(f_n^+)_*} & H_3((\BG[\SLn{n}^d])^+;\ZZ)\ar[r]^{(\iota_n^+)_*}_{\cong} & H_3(\BG[\mathrm{SL}(\CC)]^+;\ZZ)
    }
\end{equation}
Which implies that the following diagram commutes
\begin{equation*}
\xymatrix{
K_3(\CC)=\pi_3(\BG[\mathrm{SL}(\CC)]^+)\ar[r]^{h}_{\cong}\ar[rd]_{e} & H_3(\BG[\mathrm{SL}(\CC)]^+;\ZZ)\ar[r]^{(\iota_n)_*^{-1}}_{\cong} &H_3(\BG[\SLn{n}^d];\ZZ)\ar[ld]^{\widehat{\widehat{c}}_2}\\
    &  \CC/\ZZ
}
\end{equation*}
since by \eqref{eq:K3=} and \eqref{eq:c2u.ffc} we have
\begin{equation*}
 e([\Sigma,\rho])=\widehat{\widehat{c}}_2\circ(\iota_n)_*^{-1}\circ h([\Sigma,\rho])=\widehat{\widehat{c}}_2\circ(\iota_n)_*^{-1}(\langle\Sigma,\rho\rangle)=
 \widehat{\widehat{c}}_2((f_n)_*([\Sigma]))=\widehat{c}_{\rho,2}([\Sigma])=\tilde{\xi}_\rho(D).
\end{equation*}
Therefore, Theorem~\ref{th:C2.APS} is a generalization of \cite[Theorem~A]{Jones-Westbury:AKTHSEI}.

\begin{remark}
When $\rho$ is a $2$-dimensional representation ($n=2$), in diagram \eqref{eq:H.f+.e} the homomorphism $(\iota_n)_*$ is not longer an isomorphism.
We have that
\begin{equation*}
 H_3(\BG[\SLn{2}^d];\ZZ)\cong K^{\mathrm{ind}}_3(\CC)\subset K_3(\CC)\cong H_3(\BG[\mathrm{SL}(\CC)];\ZZ),
\end{equation*}
where $K^{\mathrm{ind}}_3(\CC)$ is the indecomposable part of $K_3(\CC)$, and $(\iota_n)_*$ is the inclusion. In this case we take the composition of
$f_2\colon \Sigma\to \BG[\SLn{2}]$ with the map $\BG[\SLn{2}]\to\BG[\SLn{n}]$ induced by the inclusion $\SLn{2}\to\SLn{n}$ with $n>2$.
\end{remark}

\section{CCS-numbers of \texorpdfstring{$L=\mathbb{S}^3/\Gamma$}{S3/G}}\label{sec:s3.G}

Consider the compact oriented $3$-manifolds of the form $L=\mathbb{S}^3/\Gamma$, where $\Gamma$ is a finite subgroup of $\SU$.
In this section we compute the first and second CCS-numbers of all irreducible representations $\alpha\colon\pi_1(L)=\Gamma\to \GLn{n}$.
We compare our results with the ones by C.~B.~Thomas~\cite[\S~1]{zbMATH03538557} on the Chern classes of $\alpha$.

\subsection{The 3-manifolds \texorpdfstring{$L=\mathbb{S}^3/\Gamma$}{S3/G}}

Let $\Gamma$ be a finite subgroup of $\SU$.
As usual, denote by $[\Gamma,\Gamma]$ the \textit{commutator subgroup} of $\Gamma$, let $\mathrm{Ab}\colon\Gamma\to\Gamma/[\Gamma,\Gamma]$ be the projection,
and denote by $\mathrm{Ab}(\Gamma)=\Gamma/[\Gamma,\Gamma]$, the \textit{abelianization} of $\Gamma$.

Since $\GLn{1}$ is abelian, by the universal property of the abelianization, the map $\mathrm{Ab}$ induces the isomorphism
\begin{equation}\label{eq:IsoProp.2}
 \Hom(\mathrm{Ab}(\Gamma), \GLn{1})\xrightarrow[\cong]{\mathrm{Ab}^\star}\Hom(\Gamma, \GLn{1}),
\end{equation}
between one-dimensional representations of $\Gamma$ and one-dimensional representations of $\mathrm{Ab}(\Gamma)$.
Given $\varrho \in \Hom(\Gamma, \GLn{1})$, we denote by $\varrho_{\mathrm{Ab}} \in \Hom(\mathrm{Ab}(\Gamma), \GLn{1})$ the associated representation via the isomorphism \eqref{eq:IsoProp.2}.

The group $\Gamma$ acts on $\mathbb{S}^3\cong\SU$ from the left by multiplication.
The orbit space $L=\mathbb{S}^3/\Gamma$ is a compact oriented $3$-manifold, its fundamental group is $\pi_1(L)=\Gamma$ and its homology groups are
(see \cite[p.~45]{Lamotke:RSIS})
\begin{equation*}
H_0(L;\ZZ)\cong\ZZ,\quad H_1(L;\ZZ)\cong\mathrm{Ab}(\Gamma),\quad H_2(L;\ZZ)=0,\quad\text{and}\quad H_3(L;\ZZ)\cong\ZZ.
\end{equation*}
Thus, $L$ is a rational homology $3$-sphere and we can apply the results in Subsection~\ref{ssec:cn.rhs}.
By Lemma~\ref{lem:hc1.det}, to compute the first Cheeger-Chern-Simons class $\widehat{c}_{\rho,1}$ of $\rho$ it is equivalent to compute the
first Cheeger-Chern-Simons class $\widehat{c}_{\det(\rho),1}$ of the one-dimensional representation $\det(\rho)\colon\pi_1(L)=\Gamma\to\GLn{1}$. The representation $\det(\rho)$ induces a
classifying map (with flat connection) $\bar{f}_{\det} \colon L \to \BG[\GLn{1}]^d$.
Let
\begin{equation*}
\widehat{c}_1\in \widehat{H}^{1}(\BG[\GLn{1}]^d;\CC/\ZZ)\cong\Hom(\CC^{*d},\CC/\ZZ)
\end{equation*}
be the universal first Cheeger-Chern-Simons class for flat bundles. We have
that $\widehat{c}_{\det(\rho),1}=\bar{f}_{\det}^*(\widehat{c}_1)$, hence, given a class $\bar{\nu}$ of $H_{1}(L;\ZZ)\cong\mathrm{Ab}(\Gamma)$ we have
\begin{equation*}
\widehat{c}_{\det(\rho),1}(\bar{\nu})=\bar{f}_{\det}^*(\widehat{c}_1)(\bar{\nu})=\widehat{c}_1((\bar{f}_{\det})_*(\bar{\nu})).
\end{equation*}
But the homomorphism induced in homology by the map $\bar{f}_{\det}$
\begin{equation*}
(\bar{f}_{\det})_*\colon\mathrm{Ab}(\Gamma)\cong H_{1}(L;\ZZ)\to H_{1}(\BG[\GLn{1}]^d;\ZZ)\cong\CC^{*d}
\end{equation*}
is precisely the representation $\det(\rho)_{\mathrm{Ab}}\colon\mathrm{Ab}(\Gamma)\to\GLn{1}^d\cong\CC^{*d}$. Hence we have that (see \cite[(0.2)]{DUPONT1994})
\begin{equation}\label{eq:df}
\widehat{c}_{\rho,1}(\bar{\nu})=\widehat{c}_{\det(\rho),1}(\bar{\nu})=\widehat{c}_1(\det(\rho)_{\mathrm{Ab}}(\bar{\nu}))=\frac{1}{2\pi i}\log(\det(\rho)_{\mathrm{Ab}}(\bar{\nu})).
\end{equation}

\subsection{Finite subgroups of \texorpdfstring{$\SU$}{SU(2)}}\label{ssec:Gamma}

Let $\langle 2,q,r\rangle$ denote the group given by the presentation
\begin{equation}\label{eq:2qr}
\langle 2,q,r\rangle=\langle b,c\,|\,(bc)^2=b^q=c^r\rangle
\end{equation}
and let $\zeta_{l}=e^{\frac{2\pi i}{l}}$ be the $l$-th primitive root of unity. The classification of the finite subgroups of $\SU$ is well known (see \cite{Cayley:Poly,Klein:Icosahedron,Lamotke:RSIS}).
Following, we list them together with their irreducible representations. We follow the notation of \cite{CM-TEIOTDO}, where $\chi_j$ is the character of the
representation $\alpha_j$ and $\psi_t$ is the character of the representation $\rho_t$.

\paragraph{\textbf{Cyclic groups.} $\mathsf{C}_l$, $l\geq 2$} Order $l$. We identify $\mathsf{C}_l$ with the $l$-th roots of unity. The group $\mathsf{C}_l$ has $l$
irreducible representations $\alpha_j$, $1\leq j\leq l$ given by $\alpha_j(\zeta_{l})=\zeta_l^{j-1}$.

\paragraph{\textbf{Binary dihedral groups.} $\mathsf{BD}_{2r}=\langle 2,2,r\rangle$, $r\geq2$} Order $4r$. It has $r+3$ irreducible representations. Four one-dimensional representations, denoted by $\alpha_j$, $j=1,\dots,4$ given by:
\begin{description}
 \item[$r$ even]
 \begin{align*}
\alpha_1(b)&=\alpha_1(c)=1, &\alpha_2(b)&=-1,\alpha_2(c)=1,\\
\alpha_3(b)&=1, \alpha_3(c)=-1,& \alpha_4(b)&=\alpha_4(c)=-1.
\end{align*}
\item[$r$ odd]
\begin{align*}
\alpha_1(b)&=\alpha_1(c)=1, &\alpha_2(b)&=-1,\alpha_2(c)=1,\\
\alpha_3(b)&=i, \alpha_3(c)=-1,& \alpha_4(b)&=-i,\phi_4(c)=-1.
\end{align*}
\end{description}
And $r-1$ two-dimensional representations, denoted by $\rho_t$, $1\leq t\leq r-1$, given for any $r$ by:
\begin{equation*}
\rho_t(b)=\left(\begin{smallmatrix}0&1\\(-1)^t&0\end{smallmatrix}\right),\qquad \rho_t(c)=\left(\begin{smallmatrix}\zeta_{2r}^{t}&0\\0&\zeta_{2r}^{-t}\end{smallmatrix}\right),
\end{equation*}
where $\rho_1$ is the natural representation.

\paragraph{\textbf{Binary tetrahedral group.} $\mathsf{BT}=\langle 2,3,3\rangle$} Order $24$. It has $7$ irreducible representations, denoted by $\alpha_j$, $j=1,\dots,7$.
There are three $1$-dimensional representations:
\begin{equation*}
    \alpha_1(b)=\alpha_1(c)=1, \quad \alpha_2(b)=\zeta_3 , \alpha_2(c)=\zeta_3^2 \quad \text{and} \quad \alpha_3(b)=\zeta_3^2 , \alpha_3(c)=\zeta_3.
\end{equation*}
There are three $2$-dimensional representations:
\begin{align*}
 \alpha_4(b)&=\left(\begin{smallmatrix}0& 1\\ -1& 1 \end{smallmatrix}\right), &  \alpha_4(c)&= \left(\begin{smallmatrix}0& -\zeta_3\\ \zeta_3^2& 1\end{smallmatrix}\right),\\
\alpha_5(b)&=\left(\begin{smallmatrix}-\zeta_3^2& \zeta_3\\0& -1 \end{smallmatrix}\right), &  \alpha_5(c)&= \left(\begin{smallmatrix}-1& 0\\  \zeta_3^2& -\zeta_3 \end{smallmatrix}\right),\\
 \alpha_6(b)&=\left(\begin{smallmatrix}-\zeta_3& -\zeta_3^2\\ 0& -1 \end{smallmatrix}\right), &  \alpha_6(c)&= \left(\begin{smallmatrix}0& 1\\-\zeta_3^2& \zeta_3 \end{smallmatrix}\right).\\
\intertext{where $\alpha_4$ is the natural representation. There is one $3$-dimensional representation:}
\alpha_7(b)&=\left(\begin{smallmatrix}-1& -1& -1 \\ 1& 0& 0 \\ 0& 0& 1 \end{smallmatrix}\right), &  \alpha_7(c)&=\left(\begin{smallmatrix}-1& -1& -1 \\ 0& 1& 0 \\ 1& 0& 0\end{smallmatrix}\right).
\end{align*}

\paragraph{\textbf{Binary octahedral group.} $\mathsf{BO}=\langle 2,3,4\rangle$} Order $48$. It has $8$ irreducible representations, denoted by $\alpha_j$, $j=1,\dots,8$.
There are two $1$-dimensional representations:
\begin{equation*}
\alpha_1(b)=\alpha_1(c)=1,\quad\alpha_2(b)=1,\alpha_{2}(c)=-1.
\end{equation*}
There are three $2$-dimensional representations:
\begin{align*}
\alpha_3(b)&=\left(\begin{smallmatrix}0& 1\\ -1& -1 \end{smallmatrix}\right), & \alpha_3(c)&=\left(\begin{smallmatrix}-1& -1\\ 0& 1 \end{smallmatrix}\right),\\
\alpha_4(b)&=\tfrac{1}{2}\left(\begin{smallmatrix}1+\sqrt{2}i& -i \\  -i& 1-\sqrt{2}i \end{smallmatrix}\right), & \alpha_4(c)&=\tfrac{1}{2}\left(\begin{smallmatrix}-\sqrt{2}-i& -1  \\  1& -\sqrt{2}+i \end{smallmatrix}\right),\\
\alpha_5(b)&=\left(\begin{smallmatrix} 0& -i \\  -i& 1 \end{smallmatrix}\right), & \alpha_5(c)&=\left(\begin{smallmatrix}i& -\sqrt{2}i \\ -1& \sqrt{2}-i \end{smallmatrix}\right),\\
\intertext{where $\alpha_4$ is the natural representation. There are two $3$-dimensional representations:}
\alpha_6(b)&=\left(\begin{smallmatrix}1& 0& -1 \\  0& 0& 1  \\ 0& -1& -1 \end{smallmatrix}\right), & \alpha_6(c)&=\left(\begin{smallmatrix}0& 1& 1  \\  -1& -1& 0  \\   0& 1& 0\end{smallmatrix}\right),\\
\alpha_7(b)&=\left(\begin{smallmatrix}-1& 0& -1  \\  1& 1& 0   \\   1& 0& 0 \end{smallmatrix}\right), & \alpha_7(c)&=\left(\begin{smallmatrix}0& 0& -1   \\ 1& 0& 1   \\  0& -1& 1\end{smallmatrix}\right).\\
\intertext{There is one $4$-dimensional representation:}
\alpha_8(b)&=\left(\begin{smallmatrix}-\zeta_3& 0& 0& 0   \\  -1& -1& 0& 0   \\ 0& 0& -1& -1     \\  0& 0& 0& -\zeta_3^2
\end{smallmatrix}\right), & \alpha_8(c)&=\left(\begin{smallmatrix}0& 0& \zeta_3^2& \zeta_3^2 \\ 0& 0& 0& \zeta_3   \\  0& -\zeta_3& 0& 0  \\ \zeta_3^2& -1& 0& 0
\end{smallmatrix}\right).
\end{align*}

\paragraph{\textbf{Binary icosahedral group.} $\mathsf{BI}=\langle 2,3,5\rangle$} Order $120$. It has $9$ irreducible representations. One $1$-dimensional representation $\alpha_1$;
two $2$-dimensional representations $\alpha_2,\alpha_3$, where $\alpha_2$ is the natural representation; two $3$-dimensional representations $\alpha_4,\alpha_5$;
two $4$-dimensional representations $\alpha_6,\alpha_7$; one $5$-dimensional representation $\alpha_8$ and one $6$-dimensional representation $\alpha_9$.
To save space we do not list them explicitly, since we will not need them in what follows.

\subsection{First CCS-numbers of irreducible representations of \texorpdfstring{$\pi_1(L)$}{pi1(L)}}\label{th:Gamma.C1}

Let $L=\mathbb{S}^3/\Gamma$ with $\Gamma=\mathsf{C}_l,\mathsf{BD}_{2r},\mathsf{BT},\mathsf{BO},\mathsf{BI}$.
Denote by $\bar{b},\bar{c}$ the images in $\mathrm{Ab}(\Gamma)$ of the generators $b,c$ of $\Gamma$ in presentation \eqref{eq:2qr}.
Table~\ref{tb:AbG} shows the abelianizations of $\Gamma$ (see \cite[II\S5-Table~2]{Lamotke:RSIS}).
\begin{table}[h]
\setlength{\extrarowheight}{4pt}
\begin{tabular}{|l|c|l|}\cline{1-3}
$\Gamma< \SLn{2}$ & $\mathrm{Ab}(\Gamma)$& Generators\\\hline
$\mathsf{C}_l$ &  $\mathsf{C}_l$& generated by $\zeta_l$\\\hline
$\mathsf{BD}_{2r}$ ($r$ even)&  $\mathsf{C}_2\oplus\mathsf{C}_2$& generated by $\bar{b}$ and $\bar{c}$ \\\hline
$\mathsf{BD}_{2r}$ ($r$ odd) &  $\mathsf{C}_4$& generated by $\bar{b}$ and $\bar{c}=\bar{b}^2$ \\\hline
$\mathsf{BT}$ & $\mathsf{C}_3$& generated by $\bar{c}$ and $\bar{b}=\bar{c}^{-1}$ \\\hline
$\mathsf{BO}$ & $\mathsf{C}_2$& generated by $\bar{c}$ and $\bar{b}=1$\\\hline
$\mathsf{BI}$ & $\{1\}$&\\
\hline
\end{tabular}
\caption{Abelianizations of $\Gamma$}\label{tb:AbG}
\end{table}

Let $\bar{\nu}$ be a generator of $H_{1}(L;\ZZ)=\mathrm{Ab}(\Gamma)$ (either $\zeta_l$, $\bar{b}$ or $\bar{c}$ according to $\Gamma$ in Table~\ref{tb:AbG}).
Let $\alpha\colon\pi_1(L)=\Gamma\to\GLn{n}$ be an irreducible representation. By \eqref{eq:df} and the isomorphism \eqref{eq:IsoProp.2} we have
\begin{equation*}
\widehat{c}_{\alpha,1}(\bar{\nu})=\widehat{c}_{\det\alpha,1}(\bar{\nu})=\frac{1}{2\pi i}\log(\det(\alpha(\nu))),
\end{equation*}
where $\nu=\zeta_l$ if $\bar{\nu}=\zeta_l$; $\nu=b$ if $\bar{\nu}=\bar{b}$; and $\nu=c$ if $\bar{\nu}=\bar{c}$.

Table~\ref{tb:fcn} shows the first CCS-numbers of the irreducible representations $\alpha\colon\pi_1(L)=\Gamma\to\GLn{n}$ with respect of the generators of $H_{1}(L;\ZZ)=\mathrm{Ab}(\Gamma)$.

\begin{table}[H]
\begin{equation*}
\setlength{\extrarowheight}{3pt}
\begin{array}{|c|c|c|c|c|c|c|c|c|c|}\hline
\mathsf{C}_l & \alpha_1 &  \alpha_2 & \dots & \alpha_l& & & &  & \\\hline
\widehat{c}_{\alpha,1}(\zeta_l) & 0 & \frac{1}{l} &\dots  & \frac{l-1}{l}& & & & & \\[3pt]\hline\hline
\mathsf{BD}_{2r},\ \text{$r=2l$ {\tiny even}} & \alpha_1 & \alpha_2 & \alpha_3 & \alpha_4 & \rho_{2t-1} & \rho_{2t} & & & \\\hline
\widehat{c}_{\alpha,1}(\bar{b}) & 0 & \frac{1}{2} & 0 & \frac{1}{2} & 0 & \frac{1}{2} & & & \\[3pt]\hline
\widehat{c}_{\alpha,1}(\bar{c}) & 0 & 0 & \frac{1}{2} & \frac{1}{2} & 0 & 0 & & & \\[3pt]\hline\hline
\mathsf{BD}_{2r},\ \text{$r=2l-1$ {\tiny odd}} & \alpha_1 & \alpha_2 & \alpha_3 & \alpha_4 & \rho_{2t-1} & \rho_{2t} & & & \\\hline
\widehat{c}_{\alpha,1}(\bar{b}) & 0 & \frac{1}{2} & \frac{1}{4} & \frac{3}{4} & 0 & \frac{1}{2} & & & \\[3pt]\hline
\widehat{c}_{\alpha,1}(\bar{c}) & 0 & 0 & \frac{1}{2} & \frac{1}{2} & 0 & 0 & & & \\[3pt]\hline\hline
\mathsf{BT} & \alpha_1 & \alpha_2 & \alpha_3 & \alpha_4 & \alpha_5 & \alpha_6 & \alpha_7 & & \\\hline
\widehat{c}_{\alpha,1}(\bar{c}) & 0 & \frac{2}{3} & \frac{1}{3} & 0 & \frac{1}{3} & \frac{2}{3} & 0 & & \\[3pt]\hline\hline
\mathsf{BO} & \alpha_1 & \alpha_2 & \alpha_3 & \alpha_4 & \alpha_5 & \alpha_6 & \alpha_7 & \alpha_8 & \\\hline
\widehat{c}_{\alpha,1}(\bar{c}) & 0 & \frac{1}{2} & \frac{1}{2} & 0 & 0 & \frac{1}{2} & 0 & 0 & \\[3pt]\hline\hline
\mathsf{BI} & \alpha_1 & \alpha_2 & \alpha_3 & \alpha_4 & \alpha_5 & \alpha_6 & \alpha_7 & \alpha_8 & \alpha_9 \\\hline
\widehat{c}_{\alpha,1}(1) & 0 & 0 & 0 & 0 & 0 & 0 & 0 & 0 & 0 \\\hline
\end{array}
\end{equation*}
\caption{First CCS-numbers}\label{tb:fcn}
\end{table}

\subsection{Second CCS-numbers of irreducible representations of \texorpdfstring{$\pi_1(L)$}{pi1(L)}}\label{th:C2APS}

Let $L=\mathbb{S}^3/\Gamma$ with $\Gamma=\mathsf{C}_l,\mathsf{BD}_{2r},\mathsf{BT},\mathsf{BO},\mathsf{BI}$.
To compute the second CCS-number $\widehat{c}_{\alpha,2}([L])$ of an irreducible representation $\alpha\colon\Gamma\to \GLn{n}$ we use the results of \cite{CM-TEIOTDO} as follows.

If $\alpha$ is topologically trivial ($\widehat{c}_{\alpha,1}=0$ by Lemma~\ref{lem:ci=0}), by Corollary~\ref{cor:c2.xi} we have that $\widehat{c}_{\alpha,2}([L])=\tilde{\xi}_\alpha(D)$.
If $\alpha$ is not topologically trivial ($\widehat{c}_{\alpha,1}\neq0$), by Theorem~\ref{Cor:NoTop} we have $\widehat{c}_{\alpha,2}([L])=\tilde\xi_{\alpha}(D)-\tilde\xi_{\det(\alpha)}(D)$.

\begin{proposition}\label{prop:id.c20}
The first CCS-number of one-dimensional irreducible representations $\alpha\colon\Gamma\to \GLn{1}$ of a finite subgroup $\Gamma$ of $\SU$ is zero.
\end{proposition}

\begin{proof}
From Table~\ref{tb:fcn} and by Lemma~\ref{lem:ci=0} the only $1$-dimensional irreducible representation which is topologically trivial is the
trivial representation $\alpha_1$. Thus, $\widehat{c}_{\alpha_1,2}([L])=\tilde{\xi}_{\alpha_1}(D)$ which by definition \eqref{eq:reduced} is zero.
If $\alpha$ is not topologically trivial, $\widehat{c}_{\alpha,2}([L])=\tilde\xi_{\alpha}(D)-\tilde\xi_{\det(\alpha)}(D)$, but since it is $1$-dimensional, we have that $\alpha=\det(\alpha)$ and thus $\widehat{c}_{\alpha,2}([L])=0$.
\end{proof}

By Proposition~\ref{prop:id.c20} we only need to consider the irreducible representations of dimension bigger than $1$. In Table~\ref{tb:red.xi} we recall the values of
$\tilde{\xi}_\alpha(D)$ computed in \cite{CM-TEIOTDO}.

\begin{table}[h]
\begin{equation*}
\setlength{\extrarowheight}{5pt}
\begin{array}{|*{12}{c|}}\hline
\mathsf{BT}&\mathbf{\alpha_{4}}&\alpha_{5}&\alpha_{6}&\alpha_{7}&\mathsf{BO}&\alpha_{3}&\mathbf{\alpha_{4}}&\alpha_{5}&\alpha_{6}&\alpha_{7}&\alpha_{8}\\[5pt]\hline
\tilde{\xi}_\alpha(D)&\mathbf{\frac{1}{24}}&\frac{17}{24}&\frac{17}{24}&\frac{1}{6}&\tilde{\xi}_\alpha(D)&\frac{7}{12}&\mathbf{\frac{1}{48}}&\frac{25}{48}&\frac{5}{6}&\frac{1}{12}&\frac{5}{24}\\[5pt]\hline
\mathsf{BD}_{2r}&\mathbf{\rho_1}&\mathsf{BI}&\mathbf{\alpha_{2}}&\alpha_{3}&\alpha_{4}&\alpha_{5}&\alpha_{6}&\alpha_{7}&\alpha_{8}&\alpha_{9}&\\[5pt]\hline
\tilde{\xi}_\alpha(D)&\mathbf{\frac{1}{4r}}&\tilde{\xi}_\alpha(D)&\mathbf{\frac{1}{120}}&\frac{49}{120}&\frac{19}{30}&\frac{1}{30}&\frac{5}{6}&\frac{1}{12}&\frac{1}{6}&\frac{7}{24}&\\[5pt]\hline
\end{array}
\end{equation*}
\caption{Values of $\tilde{\xi}_\alpha(D)$ from \cite{CM-TEIOTDO}. The natural representations are in boldface.}\label{tb:red.xi}
\end{table}

Table~\ref{tb:scn} shows the second CCS-numbers of the irreducible representations of $\Gamma$. For $\mathsf{BD}_{2r}$ we only give the explicit value for the natural representation $\rho_1$,
which is the one we shall use in the sequel. The values for the other $2$-dimensional representations $\rho_t$, $1<t<r$ can be computed using the formulae in \cite[Theorem~5.1]{CM-TEIOTDO} and Theorem~\ref{Cor:NoTop}.
\begin{table}[h]
\begin{equation*}
\setlength{\extrarowheight}{3pt}
\begin{array}{|c|c|c|c|c|c|c|c|c|c|}\hline
\mathsf{C}_l & \alpha_1 &  \alpha_2 & \dots & \alpha_l& & & &  & \\\hline
\widehat{c}_{\alpha,2}([L]) & 0 & 0 & \ldots & 0 & & & & & \\\hline\hline
\mathsf{BD}_{2r},\ \text{$r$ even} & \alpha_1 & \alpha_2 & \alpha_3 & \alpha_4 & \rho_1 & & & & \\\hline
\widehat{c}_{\alpha,2}([L]) & 0 & 0 & 0 & 0 & \frac{1}{4r} & & & & \\\hline
\mathsf{BD}_{2r},\ \text{$r$ odd} & \alpha_1 & \alpha_2 & \alpha_3 & \alpha_4 & \rho_1 & & & & \\\hline
\widehat{c}_{\alpha,2}([L]) & 0 & 0 & 0 & 0 & \frac{1}{4r} & & & & \\\hline\hline
\mathsf{BT} & \alpha_1 & \alpha_2 & \alpha_3 & \alpha_4 & \alpha_5 & \alpha_6 & \alpha_7 & & \\\hline
\widehat{c}_{\alpha,2}([L]) & 0 & 0 & 0 & \frac{1}{24} & \frac{3}{8} & \frac{3}{8} & \frac{1}{6} & & \\[4pt]\hline\hline
\mathsf{BO} & \alpha_1 & \alpha_2 & \alpha_3 & \alpha_4 & \alpha_5 & \alpha_6 & \alpha_7 & \alpha_8 & \\\hline
\widehat{c}_{\alpha,2}([L]) & 0 & 0 & \frac{1}{3} & \frac{1}{48} & \frac{25}{48} & \frac{7}{12} & \frac{1}{12} & \frac{5}{24} & \\[4pt]\hline\hline
\mathsf{BI} & \alpha_1 & \alpha_2 & \alpha_3 & \alpha_4 & \alpha_5 & \alpha_6 & \alpha_7 & \alpha_8 & \alpha_9 \\\hline
\widehat{c}_{\alpha,2}([L]) & 0 & \frac{1}{120} & \frac{49}{120} & \frac{19}{30} & \frac{1}{30} & \frac{5}{6} & \frac{1}{12} & \frac{1}{6} & \frac{7}{24} \\[4pt]\hline
\end{array}
\end{equation*}
\caption{Second CCS-numbers}\label{tb:scn}
\end{table}

\subsection{Comparison with the results of C.~B.~Thomas}\label{ssec:Thomas}

The group $\Gamma$ acts freely on $\mathbb{S}^3$, then $\Gamma$ has periodic cohomology of period $4$ (period $2$ if $G$ is cyclic) (see~\cite[Chapter~VI~\S~9]{zbMATH03935317}).
In Table~\ref{Tabular:CohomologyG} we summarize all the Tate cohomology groups $\widehat{H}^j(\Gamma;\ZZ)$ for $j \in \{0,1,2,3\}$.

\begin{table}[h]
\setlength{\extrarowheight}{3pt}
\begin{tabular}{|c | m{8em} | m{2cm}| m{2cm} | m{2cm} | m{2cm} |}
  \hline
$\Gamma$  & $\widehat{H}^0(\Gamma;\ZZ)$ & $\widehat{H}^1(\Gamma;\ZZ)$ & $\widehat{H}^2(\Gamma;\ZZ)$ & $\widehat{H}^3(\Gamma;\ZZ)$  \\
  \hline
  $\mathsf{C}_l$ & $\ZZ/l\ZZ$ & $0$ & $\ZZ/l\ZZ$ & $0$  \\
  \hline
  $\mathsf{BD}_{2r}$, $r$ even  & $\ZZ/4r\ZZ$ & $0$ & $\ZZ_2\oplus\ZZ_2$ & $0$  \\
  \hline
  $\mathsf{BD}_{2r}$, $r$ odd  & $\ZZ/4r\ZZ$ & $0$ & $\ZZ/4\ZZ$ & $0$  \\
  \hline
  $\mathsf{BT}$ & $\ZZ/24\ZZ$ & $0$ & $\ZZ/3\ZZ$ & $0$  \\
  \hline
  $\mathsf{BO}$ & $\ZZ/48\ZZ$ & $0$ & $\ZZ/2\ZZ$ & $0$  \\
  \hline
  $\mathsf{BI}$ & $\ZZ/120\ZZ$ & $0$ & $0$ & $0$  \\
  \hline
\end{tabular}
\caption{Tate Cohomology Groups of $\Gamma$.}\label{Tabular:CohomologyG}
\end{table}

To check that our computations are correct, we can compare the resuls on Table~\ref{tb:scn} with the ones given by C. B. Thomas~\cite[\S~1]{zbMATH03538557}, where the author
expreses the total Chern classes of the irreducible representations of $\Gamma$, a finite subgroup of $\SU$, in terms of generators of the cohomology ring of the classifying space $\BG[\Gamma]$.

Let $L=\mathbb{S}^3/\Gamma$ with $\Gamma=\mathsf{C}_l,\mathsf{BD}_{2r},\mathsf{BT},\mathsf{BO},\mathsf{BI}$. Consider the map $\phi\colon L\to \BG[\Gamma]$ and let $\varsigma=\phi_*([L])$, the image of the fundamental class $[L]\in H_3(L;\ZZ)$ by the homomorphism $\phi_*\colon H_3(L;\ZZ)\to H_3(\BG[\Gamma];\ZZ)$ induced by $\phi$ in homology. By \eqref{eq:ccsL.ccsBG}, the second CCS-number of a representation $\alpha\colon\Gamma\to\GLn{n}$ is given by
\begin{equation}\label{eq:ccs.BG.varsigma}
\widehat{c}_{\alpha,2}([L])=\widehat{c}_2(\alpha)(\varsigma),
\end{equation}
where $\widehat{c}_2(\alpha)\in H^3(\BG[\Gamma];\CC/\ZZ)$ is the second Cheeger-Chern-Simons class of $\alpha$. Using the isomorphism \eqref{eq:Iso.Hom.Chern}
$H^3(\BG[\Gamma];\CC/\ZZ)\cong \Hom(H_{3}(\BG[\Gamma];\ZZ),\CC/\ZZ)$ we can define a homomorphism
\begin{equation}\label{eq:ev.sigma}
\begin{split}
H^3(\BG[\Gamma];\CC/\ZZ)&\to\CC/\ZZ\\
\nu&\mapsto \nu(\varsigma).
\end{split}
\end{equation}
by evaluation in the homology class $\varsigma\in H_3(\BG[\Gamma];\ZZ)$. By the Universal Coefficients Theorem, the fact that $\CC$ is a divisible abelian group and that
the homology groups $H_j(\BG[\Gamma];\ZZ)$ are torsion, one can prove that $H^j(\BG[\Gamma];\CC)=0$ for all positive integer $j$.
Using this, in the cohomology long exact sequence \eqref{eq:coef.CheegerSimons2} of $\BG$ corresponding to the short exact sequence \eqref{eq:coef.CheegerSimons1}, we get an isomorphism
$H^{j}(\BG[\Gamma];\ZZ)\cong H^{j-1}(\BG[\Gamma]; \CC/\ZZ)$ for $j\geq 2$, under which, the $k$-th Chern class $c_k(\alpha)$ and the $k$-th Cheeger-Chern-Simons class $\widehat{c}_{k}(\alpha)$ of $\alpha$ correspond to each other
\begin{equation}\label{eq:crk.ckr}
\begin{split}
H^{j}(\BG[\Gamma];\ZZ)&\cong H^{j-1}(\BG[\Gamma]; \CC/\ZZ)\\
c_k(\alpha)&\leftrightarrow \widehat{c}_k(\alpha).
\end{split}
\end{equation}
Composing isomorphism \eqref{eq:crk.ckr} for $j=4$ with homomorphism \eqref{eq:ev.sigma} we get a homomorphism
\begin{equation}\label{eq:hom.c.cn}
 H^{4}(\BG[\Gamma];\ZZ)\cong H^{3}(\BG[\Gamma]; \CC/\ZZ)\to\CC/\ZZ,
\end{equation}
which sends the second Chern class $c_2(\alpha)$ of a representation $\alpha$ to its second CCS-number $\widehat{c}_2(\alpha)(\varsigma)$.
Thus, any relation between cohomology classes in $H^{4}(\BG[\Gamma];\ZZ)$ will also be satisfied by the corresponding images in $\CC/\ZZ$.

For instance, in the case of the binary tetrahedral group $\mathsf{BT}$, we have \cite[\S~1]{zbMATH03538557}:
\begin{equation}\label{eq:ch.ring}
H^\bullet(\mathsf{BT};\ZZ)=\ZZ_8[x]+\ZZ_3[y],
\end{equation}
with $x\in H^4(\mathsf{BT};\ZZ)$ and $y\in H^2(\mathsf{BT};\ZZ)$. The total Chern classes of the irreducible representations of $\mathsf{BT}$ given in Subsection~\ref{ssec:Gamma} are
\begin{align}
c(\alpha_1)&=1, &   c(\alpha_2)&=1-y,\notag\\
c(\alpha_3)&=1+y, & c(\alpha_4)&=1+x-y^2,\label{eq:c2.a4}\\
c(\alpha_5)&=1+x-y, & c(\alpha_6)&=1+x+y,\label{eq:c2.a56}\\
c(\alpha_7)&=1+4x-y^2.\label{eq:c2.a7}
\end{align}
Recall that the natural representation is $\alpha_4$, its second Chern class $c_2(\alpha_4)$ generates $H^4(\mathsf{BT};\ZZ)$ \cite[Proposition~2]{zbMATH03538557}.
Therefore, for any representation $\alpha\colon \Gamma \to \mathrm{GL}(n,\CC)$, there exists $N\in \ZZ$ such that $\widehat{c}_2(\alpha)=N\widehat{c}_2(\alpha_4)$.

From Table~\ref{tb:scn} we have
\begin{equation}\label{eq:rel1.hc}
 9\cdot\widehat{c}_2(\alpha_4)(\varsigma)=9\cdot\frac{1}{24}=\frac{9}{24}=\widehat{c}_2(\alpha_5)(\varsigma)=\widehat{c}_2(\alpha_6)(\varsigma)
\end{equation}
and
\begin{equation}\label{eq:rel2.hc}
 4\cdot\widehat{c}_2(\alpha_4)(\varsigma)=4\cdot\frac{1}{24}=\frac{4}{24}=\widehat{c}_2(\alpha_7)(\varsigma).
\end{equation}
On the other hand, from \eqref{eq:c2.a4} we have that $c_2(\alpha_4)=x-y^2$. By \eqref{eq:ch.ring} and \eqref{eq:c2.a56} we have
\begin{equation}\label{eq:rel1.c}
9\cdot c_2(\alpha_4)=9(x-y^2)=9x-9y^2=x=c_2(\alpha_5)=c_2(\alpha_6),
\end{equation}
and by \eqref{eq:c2.a7} we have
\begin{equation}\label{eq:rel2.c}
4\cdot c_2(\alpha_4)=4(x-y^2)=4x-4y^2=4x-y^2=c_2(\alpha_7).
\end{equation}
Hence, relations \eqref{eq:rel1.c} and \eqref{eq:rel2.c} correspond, respectively, to relations  \eqref{eq:rel1.hc} and \eqref{eq:rel2.hc} under homomorfism \eqref{eq:hom.c.cn}.
The other cases are analogous.

\section{Applications to singularity theory}
In this section, motivated by the results of Section~\ref{sec:s3.G}, we define new invariants for surface singularities.
The main result of the section recovers the spectrum of rational double point singularities, from the CCS-numbers of the irreducible representations of its local fundamental group.
Finally, using a result by Atiyah, Patodi and Singer we show other ways to compute the invariant $\tilde{\xi}_{\rho}(D)$ with $\rho$ a representation of the local fundamental group
of a normal surface singularity with link a rational homology sphere, using a resolution or an smoothing of the singularity.

\subsection{Normal surface singularities} 

We recall basic facts about normal surface sigularities, for further reference see \cite{Nemethi:NSS}.
Denote by $(X,x)$ either a complex analytic normal  surface germ or the spectrum of a normal complete $\mathbb{C}$-algebra of dimension $2$.
Let $\pi \colon \Rs\to X$, be a \emph{resolution of} $(X,x)$, i.e., a proper holomorphic map from a smooth surface $\Rs$ to a given representative of $(X,x)$ such that $\pi$ is biholomorphic in the complement of $\pi^{-1}(x)$. Denote the exceptional divisor by $E:=\pi^{-1}(x)$, with irreducible components $E_1,\dots,E_s$. The resolution is \emph{minimal} if there is no rational irreducible exceptional divisor $E_i$ with self intersection $E_i^2=-1$. Fix any resolution $ \pi \colon \Rs\to X$.  The group of \emph{divisors} $\mathrm{Div}(\Rs)$ of $\Rs$ is defined as
\begin{equation*}
    \mathrm{Div}(\Rs):= \set{\sum n_j D_j}{\text{$D_j$ is a irreducible curve on $\Rs$ and $n_j \in \ZZ$}}.
\end{equation*}
The \emph{support} of a divisor $D=\sum n_j D_j$, denoted by $|D|$, is the union of the irreducible curves $D_j$; i.e., $|D|=\bigcup D_j$.
\begin{definition}
    Let $D$ be a divisor. We say that:
    \begin{enumerate}
        \item  $D$ is a \emph{cycle}  if $|D| \subset E$.
        \item $D$ is a \emph{rational cycle}  if $D=\sum r_j E_j$ where $r_j \in \QQ$.
    \end{enumerate}
    An integer or rational cycle is  \emph{effective}, if $r_j \geq 0$ for all $j$.
\end{definition}
There is a natural ordering of the cycles: Let $Z'=\sum n'_j E_j$ and $Z=\sum n_j E_j$ two cycles. We say that $Z' \leq Z$ if and only if $n'_j \leq n_j$ for any $j$. Moreover, there is also a notion of intersection of cycles given as follows:
\begin{equation*}
    Z' \cdot Z = \left ( \sum n'_i E_i \right) \cdot \left ( \sum n_j E_j \right) = \sum_{i,j} n_i n_j \left( E_i \cdot E_j \right).
\end{equation*}
Denote by 
\begin{equation*}
    \mathcal{Z}_{\text{top}} := \set{Z \neq 0 \, \text{effective cycle}}{ Z \cdot E_j \leq 0 \, \text{for all $j$}}.
\end{equation*}
It is well known that $\mathcal{Z}_{\text{top}}$ has a unique minimal element $Z_{\text{fund}}$, called \emph{the fundamental cycle} (sometimes also called  \emph{the minimal cycle or Artin's fundamental cycle}). See~\cite[Chapter~2]{zbMATH01444575} for more details.

\begin{definition}\label{def:pg}
The {\em geometric genus} $p_g$ of $X$ is defined to be the dimension as a $\mathbb{C}$-vector space of $H^1(\Rs,\Ss{\Rs})$. It is independent of the choice of the resolution.

We say that $(X,x)$  has a \emph{rational singularity}, if the geometric genus of $X$ is zero.
\end{definition}

A \emph{quotient surface singularity} $(X,x)$ is isomorphic to a germ $(\CC^2/\Gamma,0)$, where $\Gamma$ is a small\footnote{with no (pseudo)reflexions. } finite subgroup of $\GLn{2}$.
If $\Gamma$ and $\Gamma'$ are two small finite subgroups of $\GLn{2}$, then $\CC^2/\Gamma$ and $\CC^2/\Gamma'$ are analytically isomorphic if $\Gamma$ and $\Gamma'$ are conjugate.     Any quotient surface singularity is a rational normal surface singularity
 (for more details see \cite[Chapter 7, Section 4]{zbMATH06313585}).

Let $X$ be a representative of the germ $(X,x)$, we may assume that $X\setminus\{x\}$ is connected and $(X,x) \subset (\CC^n,0)$. If $\epsilon>0$ is small enough, the intersection $L=X \cap \mathbb{S}_{\epsilon}^{2n-1}$ is called the \emph{link} of $(X,x)$. Recall that the link does not depend on $\epsilon$.  Moreover, the  link $L$ is a smooth, compact, connected and oriented $3$-manifold.
In singularity theory is common to denote $\pi_1^{\mathrm{loc}}(X,x)=\pi_1(L)$ and call it the \emph{local fundamental group}.
In the case of a normal surface singularity, the local fundamental group satisfies the following characterization.
\begin{proposition}[\cite{MR3323576}]
\label{prop:FundGroupLink}
For a normal surface singularity $(X,x)$,
\begin{enumerate}
\item The  group $\pi_1^{\mathrm{loc}}(X,x)$ is finite if and only if $(X,x)$ is a quotient singularity.
\item If $\pi_1^{\mathrm{loc}}(X,x)$ is infinite, then the link $L$ is an Eilenberg–MacLane space of type $K(\pi_1(L),1)$. Furthermore, the local fundamental group is torsion free.\label{it:infinite}
\end{enumerate}
\end{proposition}

\begin{remark}
\label{remark:linkrhs}
    The link of any normal rational surface singularity is a rational homology sphere (see~\cite[Exercise~1.27 and \S~3.9]{zbMATH01444575}). Moreover, the exceptional divisor of their minimal resolution is the union of finitely many components $E_i$ isomorphic to projective lines $\mathbb{CP}^1$, where $E_i^2 \leq -2$ (see~\cite[\S~3.9]{zbMATH01444575}).
\end{remark}

 A \emph{rational double point singularity} $(X,x)$ is a quotient singularity $(\CC^2/\Gamma,0)$ where $\Gamma$ is a finite subgroup of $\SLn{2}$.
Every finite subgroup of $\SLn{2}$ is small and it is conjugate to a subgroup of $SU(2)$, thus, rational double point singularities are isomorphic to germs of the form $(\CC^2/\Gamma,0)$ with $\Gamma=\mathsf{C}_l,\mathsf{BD}_{2r},\mathsf{BT},\mathsf{BO},\mathsf{BI}$, and their links are the rational homology $3$-spheres $L=\mathbb{S}^3/\Gamma$ mentioned in Section~\ref{sec:s3.G}.
Rational double point singularities are the only quotient surface singularities that embed in $\CC^3$.
For more details on rational double point singularities see for instance \cite{Lamotke:RSIS}.

\begin{remark}
In the case of rational double point singularities, the fundamental cycle is well known (see~\cite[Example~7.2.5]{zbMATH06313585}). Table~\ref{tb:zfund} contains the coefficients (without any particular ordering) of the fundamental cycle for rational double point singularities.
\begin{table}[h]
\setlength{\extrarowheight}{3pt}
\begin{tabular}{|c | c |}
\hline
Singularity  & Coefficients of $Z_{\text{fund}}$ \\
\hline
$\CC^2/\mathsf{C}_l$ &  $1, 1, \dots, 1$ {\tiny ($l-1$ times)}\\[3pt]
\hline
$\CC^2/\mathsf{BD}_{2r}$ &  $1,1,1$ and $2, \dots, 2$  {\tiny ($r-1$ times)}  \\[3pt]
\hline
$\CC^2/\mathsf{BT}$ & $1,1,2,2,2,3$  \\[3pt]
\hline
$\CC^2/\mathsf{BO}$ & $1,2,2,2,3,3,4$ \\[3pt]
\hline
$\CC^2/\mathsf{BI}$ & $2,2,3,3,4,4,5,6$  \\[3pt]
\hline
\end{tabular}
\caption{Coefficients of $Z_{\text{fund}}$ of rational double point singularities}\label{tb:zfund}
\end{table}
\end{remark}

Let $\mathcal{O}=\mathcal{O}_{\CC^{n+1},0}$ be the ring of germs of holomorphic functions and let $f\in\mathcal{O}$ with an isolated critical point at $0$.
Let $\mu=\mu(f)$ be the Milnor number of $f$ and let $(X,0)$ be the germ of the isolated hypersurface singularity defined by $f$, i.~e., $X=f^{-1}(0)$.
The \emph{spectrum} $sp(X)$ of $(X,0)$ is a notion introduced by Steenbrink \cite{Steenbrink:SHS}
using the mixed Hodge structure on the cohomology of the Milnor fibre associated to a complex hypersurface singularity together with the monodromy.
It is a (multi)-set of $\mu$ rational numbers called
spectral numbers: $sp(X)=\{\alpha_1\leq\alpha_2\leq\alpha_3\leq\dots\leq\alpha_\mu\}$.
The spectrum is an invariant of a singularity: let $f,g\in\mathcal{O}$, $X=f^{-1}(0)$ and $Y=g^{-1}(0)$; if $f$ and $g$ are \emph{right equivalent} $f\sim g$, then $sp(X)=sp(Y)$.
In fact, if $f \sim u\cdot g$, where $u$ is a unit of $\mathcal{O}$, we say that $f$ and $g$ are \emph{contact equivalent} and then $sp(X)=sp(Y)$.
The spectrum has a lot of deeper properties, for instance: if $\alpha\in sp(X)$, then $\lambda:= \exp(2\pi i\alpha)$
is an eigenvalue of the cohomological monodromy transformation.
In other words, the spectral numbers $\alpha_1,\dots,\alpha_\mu$ are specific logarithms of the monodromy eigenvalues.
The spectra of rational doble point singularities are given in Table~\ref{tb:spf}, where in the pair $(s,m)$, $s$ is the spectral number and $m$ its multiplicity see \cite[p.~165]{Steenbrink:SHS}.
\begin{table}[H]
\setlength{\extrarowheight}{3pt}
\begin{tabular}{|l | l |}
\hline
Singularity  & Spectrum \\
\hline
$\CC^2/\mathsf{C}_l$ &  $(\frac{1}{l},1),\dots, (\frac{l-1}{l},1)$ \\[3pt]\hline
$\CC^2/\mathsf{BD}_{2r}$, $r=2l$ (even)&  $(\frac{1}{2(r+1)},1),\dots,(\frac{2l-1}{2(r+1)},1),(\frac{1}{2},2),(\frac{2l+3}{2(r+1)},1),\dots,(\frac{2r+1}{2(r+1)},1)$  \\[3pt]\hline
$\CC^2/\mathsf{BD}_{2r}$, $r=2l-1$ (odd)&  $(\frac{1}{2(r+1)},1),\dots,(\frac{2l-1}{2(r+1)},1),(\frac{1}{2},1),(\frac{2l+3}{2(r+1)},1),\dots,(\frac{2r+1}{2(r+1)},1)$  \\[3pt]\hline
$\CC^2/\mathsf{BT}$ & $(\frac{1}{12},1),(\frac{4}{12},1),(\frac{5}{12},1),(\frac{7}{12},1),(\frac{8}{12},1),(\frac{11}{12},1)$  \\[3pt]\hline
$\CC^2/\mathsf{BO}$ & $(\frac{1}{18},1),(\frac{5}{18},1),(\frac{7}{18},1),(\frac{9}{18},1),(\frac{11}{18},1),(\frac{13}{18},1),(\frac{17}{18},1)$ \\[3pt]\hline
$\CC^2/\mathsf{BI}$ & $(\frac{1}{30},1),(\frac{7}{30},1),(\frac{11}{30},1),(\frac{13}{30},1),(\frac{17}{30},1),(\frac{19}{30},1),(\frac{23}{30},1),(\frac{29}{30},1)$  \\[3pt]\hline
\end{tabular}
\caption{Spectrum of rational double point singularities}\label{tb:spf}
\end{table}

\subsection{Recovering the spectrum of rational double point singularities from the CCS-numbers}

Let $(X,x)$ be a normal surface singularity. Its link $L$ is a compact oriented $3$-manifold, denote by $\Gamma=\pi_1(L)$ its fundamental group.
Let $\rho \colon \Gamma \to \GLn{n}$ be a representation.
By \eqref{eq:ccs.BG.varsigma} the second CCS-number $\widehat{c}_{\rho,2}([L])$ of $\rho$ can be computed as $\widehat{c}_2(\rho)(\varsigma)$. By Remark~\ref{rem:G.M.ccs} the Cheeger-Chern-Simons class $\widehat{c}_2(\rho)$ only depends on the representation $\rho$, and the homology class $\varsigma$ only depends on $L$ (see Subsection~\ref{ssec:Thomas}).
Thus, all the second CCS-numbers $\widehat{c}_{\rho,2}([L])$ can be determined using cohomological information of the classifying space and the homology class $\varsigma$ coming from $L$.
We would like to obtain an invariant coming from the topology of $L$ and from the singularity. For this, we introduce a variation of the second CSS-number $\widehat{c}_{\rho,2}([L])$
for normal surface singularities which takes into account information from a resolution of $X$.

\begin{definition}
\label{def:Xi.Fund}
Let $(X,x)$ be a normal surface singularity and $L$ its link. Let $\pi \colon \Rs \to X$ be a resolution and $Z_{\text{fund}}=\sum n_iE_i$ its fundamental cycle.
Let $\rho \colon\pi_1(L) \to GL(n,\CC)$ be a representation. The invariant $\Xi_{\rho}(X,\Rs)$ is given by
\begin{equation*}
   \Xi_{\rho}(X,\Rs):= \left(\frac{1+\sum n_i^2}{ 1+\sum n_i}\right)\widehat{c}_{\rho,2}([L]).
\end{equation*}
\end{definition}

\begin{remark}
Note that the definition of  $\Xi_{\rho}(X,\Rs)$ depends on the resolution.
\end{remark}

Our main result in this section is the following Theorem~\ref{th:C1C2SonSpf}, which tells us how to recover the spectrum of a rational double point singularity
from the first CCS-numbers of the irreducible representations of its local fundamental group and the invariant $\Xi_{\rho}(X,\Rs)$ for the minimal resolution and
the natural representation.

\begin{theorem}
\label{th:C1C2SonSpf}
Let $(X,x)\cong(\CC^2/\Gamma,0)$ be the germ of a rational double point surface singularity. Let $\pi \colon \Rs_{\mathrm{min}} \to X$ be the minimal resolution.
Then, the spectrum of $(X,x)$ and its multiplicities can be obtained from the non-zero first CCS-numbers of the irreducible representations of $\Gamma=\pi_1^{\mathrm{loc}}(X,x)$ and
$\Xi_{\alpha_{\mathrm{Nat}}}(X,\Rs_{\mathrm{min}})$, where $\alpha_{\mathrm{Nat}} \colon \Gamma \to \mathrm{SL}(2,\CC)$ is the natural representation.
\end{theorem}
\begin{proof}
Let $\pi\colon \Rs_{\mathrm{min}} \to X$ be the minimal resolution. Denote by $Z_{\text{fund}}=\sum n_iE_i$ its fundamental cycle. Denote by $t_\Gamma =\frac{1+\sum n_i^2}{ 1+\sum n_i}$.
Let $\alpha_{\mathrm{Nat}} \colon \Gamma \to \mathrm{SL}(2,\CC)$ be the natural representation.
In the following table we have computed the rational number $t_\Gamma$ and the invariant $\Xi_{\alpha_{\mathrm{Nat}}}(X,\Rs_{\mathrm{min}})$ from Table~\ref{tb:zfund} and Table~\ref{tb:red.xi} respectively:
\begin{table}[h]
\setlength{\extrarowheight}{3pt}
\begin{tabular}{ |c |c| c | c | c | c | c | c|}
  \hline
  & $\mathsf{BD}_{2r}$ & $\mathsf{BT}$ & $\mathsf{BO}$ & $\mathsf{BI}$\\
  \hline
  $t_\Gamma$& $\frac{2r}{r+1}$  & $\frac{24}{12}=2$ & $\frac{48}{18}$ &  $\frac{120}{30}=4$ \\ \hline
  $\Xi_{\alpha_{\mathrm{Nat}}}(X,\Rs_{\mathrm{min}})$& $\frac{1}{2(r+1)}$  & $\frac{1}{12}$ & $\frac{1}{18}$ &  $\frac{1}{30}$ \\ \hline
\end{tabular}
\end{table}

We can recover the spectral numbers given in Table~\ref{tb:spf} as follows:
\begin{enumerate}
\item There are two cases:
\begin{description}
 \item[$\Gamma=\mathsf{C}_l$] Each different non-zero first CCS-number of an irreducible representation of $\mathsf{C}_l$, with respect to the generator $\zeta_l$ of $H_1(\BG[\mathsf{C}_l];\ZZ)$, is a spectral number with multiplicity $1$.
 \item[$\Gamma\neq\mathsf{C}_l$] Each different non-zero first CCS-number of an irreducible representation of $\Gamma$, with respect to the class $\bar{c}$ of $H_1(\BG[\Gamma];\ZZ)$, is a spectral number with multiplicity $1$ (repeated first CCS-numbers are counted just one time).
 \end{description}
 \item If $\Xi_{\alpha_{\mathrm{Nat}}}(X,\Rs_{\mathrm{min}})\neq0$, it is of the form $\frac{1}{m}$ for some integer $m$ and it is a spectral number with multiplicity $1$.
 \item One needs to consider two cases:
 \begin{description}
  \item[$\mathsf{BD}_{2r}$] The multiples $(2i+1)\cdot\Xi_{\alpha_{\mathrm{Nat}}}(X,\Rs_{\mathrm{min}})$ with $i=1,\dots,r$ are the rest of the spectral numbers, each one with multiplicity $1$.
  \item[$\mathsf{BT}$, $\mathsf{BO}$ or $\mathsf{BI}$] The multiples $k\cdot\Xi_{\alpha_{\mathrm{Nat}}}(X,\Rs_{\mathrm{min}})$ with $1<k<m$ such that $k$ and $m$ are coprime, are the rest of the spectral numbers, each one with multiplicity $1$.
 \end{description}
\end{enumerate}
Let us check this case by case.
\begin{description}
\item[Case $\mathsf{C}_l$] From Table~\ref{tb:scn} the first CCS-numbers of non-trivial irreducible representations of $C_l$, with respect to the generator $\zeta_l$, are $\frac{1}{l},\dots, \frac{l-1}{l}$, which are precisely the spectrum. All of them with multiplicity $1$.
\item[Case $\mathsf{BD}_{2r}$]From Table~\ref{tb:scn} the non-zero first CCS-numbers of the irreducible representations of $\mathsf{BD}_{2r}$, with respect to $\bar{c}$,
are equal to $\frac{1}{2}$, which we take with multiplicity $1$.
Also the number $\Xi_{\alpha_{\mathrm{Nat}}}(X,\Rs_{\mathrm{min}})=\frac{1}{2(r+1)}$ is a spectral number with multiplicity $1$.
Taking the multiples $(2i+1)\cdot\Xi_{\alpha_{\mathrm{Nat}}}(X,\Rs_{\mathrm{min}})$ with $i=1,\dots,r$ we get: $\frac{3}{2(r+1)},\frac{5}{2(r+1)},\dots,\frac{2r+1}{2(r+1)}$ which are the rest of the spectral numbers, all taken with multiplicity $1$. Thus, all the spectral numbers have multiplicity $1$, except for the case $r$ even ($r=2l$), where the number $\frac{1}{2}$ has multiplicity $2$, one coming
from the first CSS-numbers, and the other one from $\frac{2l+1}{2(r+1)}=\frac{r+1}{2(r+1)}=\frac{1}{2}$.
\item[Case $\mathsf{BT}$] From Table~\ref{tb:scn} the non-zero first CCS-numbers of the irreducible representations of $\mathsf{BT}$, with respect to $\bar{c}$, are $\frac{4}{12}$ and $\frac{8}{12}$, which are spectral numbers, each one taken with multiplicity $1$.
Also the number $\Xi_{\alpha_{\mathrm{Nat}}}(X,\Rs_{\mathrm{min}})=\frac{1}{12}$ is a spectral number with multiplicity $1$. Taking the multiples $k\cdot\frac{1}{12}$ such that $1<k<12$ and $k$ is coprime with $12$ we get: $\frac{5}{12},\frac{7}{12},\frac{11}{12}$ which are the rest of the spectral numbers, all taken with multiplicity $1$.
Thus, all the spectral numbers have multiplicity $1$.
\item[Case $\mathsf{BO}$]
From Table~\ref{tb:scn} the non-zero first CCS-numbers of the irreducible representations of $\mathsf{BT}$, with respect to $\bar{c}$, are equal to $\frac{1}{2}=\frac{9}{18}$ which is a spectral number, which we take with multiplicity $1$.
Also $\Xi_{\alpha_{\mathrm{Nat}}}(X,\Rs_{\mathrm{min}})=\frac{1}{18}$ is a spectral number with multiplicity $1$. Taking the multiples $k\cdot\frac{1}{18}$ such that $1<k<18$ and $k$ is coprime with $18$ we get: $\frac{5}{18},\frac{7}{17},\frac{11}{18},\frac{13}{18},\frac{17}{18}$ which are the rest of the spectral numbers, all taken with multiplicity $1$.
Thus, all the spectral numbers have multiplicity $1$.
\item[Case $\mathsf{BI}$] All the first CCS-numbers are zero. $\Xi_{\alpha_{\mathrm{Nat}}}(X,\Rs_{\mathrm{min}})=\frac{1}{30}$ is a spectral number with multiplicity $1$.
Taking the multiples $k\cdot\frac{1}{30}$ such that $1<k<30$ and $k$ is coprime with $30$ we get: $\frac{7}{30},\frac{11}{30},\frac{13}{30},\frac{17}{30},\frac{19}{30},\frac{23}{30},\frac{29}{30}$ which are all the spectral numbers, taken with multiplicity $1$.
\end{description}
This proves the assertion.
\end{proof}

\begin{remark}
Notice that for $\Gamma=\mathsf{BT},\mathsf{BO},\mathsf{BI}$ and $\Gamma=\mathsf{BD}_{2r}$ with $r$ even, the homology class $\bar{c}$ is a generator of $H_1(\BG[\Gamma];\ZZ)$ except for the case of
$\Gamma=\mathsf{BD}_{2r}$ with $r$ odd. If in this case we take the first CCS-numbers with respect to the generator $\bar{b}$, we get the numbers $\frac{1}{4}$ and $\frac{3}{4}$
which are not spectral numbers.
\end{remark}

\begin{remark}\label{rk:Xi.xi}
If $(\CC^2/\Gamma,0)$ is a rational double point singularity with link $L$, from Table~\ref{tb:fcn} the natural representation $\alpha_{\mathrm{Nat}} \colon \Gamma \to \mathrm{SL}(2,\CC)$ is topologically trivial.
By Theorem~\ref{th:C2.APS} $\widehat{c}_{\alpha_{\mathrm{Nat}},2}([L])=\tilde{\xi}_{\alpha_{\mathrm{Nat}}}(D)$, where $D$ is the Dirac operator of $L$.
Thus $\Xi_{\alpha_{\mathrm{Nat}}}(X,\Rs)=\left(\frac{1+\sum n_i^2}{ 1+\sum n_i}\right)\tilde{\xi}_{\alpha_{\mathrm{Nat}}}(D)$.

A natural question is if one gets interesting invariants replacing in $\Xi_{\rho}(X,\Rs)$ the Dirac operator by other (self-adjoint) differential operators.
\end{remark}

\subsection{McKay correspondence}
In this section we recall basics on McKay correspondence. We assume basic familiarity with
dualizing sheaves, modules and normal surface singularities,  see~\cite{zbMATH01194481,zbMATH01444575,zbMATH06313585} for more details.

Let $X$ be a normal variety. Let $\Homs_{\Ss{X}}(\bullet,\bullet)$ and $\Exts^{\ i}_{\Ss{X}}(\bullet,\bullet)$ be the sheaf theoretic Hom and Ext functors (see \cite{Hartshorne:AlgGeom}). 
An $\Ss{X}$-module $M$ is {\em indecomposable} if it cannot be written as a direct sum of two non-trivial submodules.
The dual of an $\Ss{X}$-module $M$ is denoted by $M^{\smvee}:=\Homs_{\Ss{X}}(M,\Ss{X})$. An $\Ss{X}$-module $M$ is called \emph{reflexive} if the natural homomorphism from $M$ to $M^{\smvee \smvee}$ is an isomorphism. 

Let $(X,x)$ be the germ of a normal surface singularity and $\pi \colon \Rs \to X$
be a resolution. Recall the following definition of full sheaves as in~\cite{zbMATH04081723}.
\begin{definition}
An $\Ss{\Rs}$-module $\Sf{M}$ is called \emph{full} if there is a reflexive $\Ss{X}$-module $M$ such that 
$\Sf{M} \cong \left(\pi^* M\right)^{\smvee \smvee}$. We call $\Sf{M}$ the full sheaf associated to $M$.
\end{definition}

By Artin and Verdier \cite{Artin-Verdier:RMORDP}, the classical \emph{McKay correspondance} can be stated as follows:
for rational double point singularities there is a one-to-one correspondence between,
\begin{itemize}
    \item non-trivial irreducible representations of $\pi_1^{\mathrm{loc}}(X,x)$,
    \item  non-trivial indecomposable reflexive $\Ss{X}$-modules,
    \item irreducible components of the exceptional divisor of the minimal resolution of $X$.
\end{itemize}
\begin{remark}
    In general, given any $(X,x)$ normal surface singularity and $\rho$ a representation of $\pi_1^{\mathrm{loc}}(X,x)$, by the Riemann-Hilbert correspondence we obtain a reflexive $\Ss{X}$-module. Such a module depends on the representation, therefore we will denote it by $M_\rho$. The reader may see \cite{GUSTAVSEN2008851} for more details.
\end{remark}
If the singularity is not a rational double point singularity, we need to impose an additional hypothesis to produce a classification analogous to the McKay correspondence. For rational singularities this property is called \emph{speciality}, this has been done by several people: Esnault~\cite{zbMATH03880862}, Wunram\cite{zbMATH03997956}, Riemenschneider~\cite{10.14492/hokmj/1350657526}, Iyama and Wemyss~\cite{zbMATH05700565}. For Gorenstein singularities, it is called \emph{cohomological speciality}
by Fernández de Bobadilla and Romano-Velázquez~\cite{BoRo}.

\begin{definition}
Let $(X,x)$ be the germ of a normal surface singularity. Let $M$ be a reflexive $\Ss{X}$-module of rank $r$. We say that $M$ is
\begin{enumerate}
    \item \emph{special} if $\Ext_{\Ss{X}}^1(M,\Ss{X})=0$,
    \item \emph{cohomologically special} if for any resolution $\pi \colon \Rs \to X$ the full sheaf $\Sf{M}$ associated to $M$ satifies $\dimc{R^1 \pi_* \left(\Sf{M}^{\smvee}\right)} = rp_g$, where $p_g$ is the geometric genus (Definition~\ref{def:pg}).
\end{enumerate}
\end{definition}

For representations of the local fundamental group, we have the following definition.
\begin{definition}
    Let $(X,x)$ be a normal surface singularity and $\rho$ be a representation of $\pi_1^{\mathrm{loc}}(X,x)$. Let $M_\rho$ be the reflexive $\Ss{X}$-module given by $\rho$ and the Riemann-Hilbert correspondence. We say that the representation $\rho$ is
    \begin{enumerate}
        \item \emph{special} if $M_\rho$ is special.
        \item \emph{cohomologically special} if $M_\rho$ is cohomologically special.
    \end{enumerate}
\end{definition}

\begin{remark}\label{remark:specialRDP}
    By~\cite{zbMATH03997956}, in rational double point singularities any representation (reflexive module) is a special representation (special reflexive module). Moreover, by~\cite{BoRo} in this case the notions of special and cohomologically special agree.
\end{remark}

Let $(X,x)$ be the germ of a rational double point surface singularity and consider the natural representation $\alpha_{\mathrm{Nat}} \colon \pi_1^{\mathrm{loc}}(X,x) \to SL(2,\CC)$.
By Theorem~\ref{th:C1C2SonSpf} $\Xi_{\alpha_{\mathrm{Nat}}}(X,\Rs_{\mathrm{min}})$ is always a spectral number.
Furthermore, for any irredicible representation $\alpha \colon \pi_1^{\mathrm{loc}}(X,x) \to \GLn{n}$
if its first Cheeger-Chern-Simons class $\widehat{c}_1(\alpha)$ is not zero, then the corresponding first CCS-number $\widehat{c}_1(\alpha)(\bar{c})$ is always a spectral number.
We would like to define a new invariant for rational or Gorenstein surface singularities $(X,x)$ using the Cheeger-Chern-Simons classes $\widehat{c}_i(\alpha)$ with $i=1,2$ of the
irreducible representations $\text{Irr}(\pi_1^{\mathrm{loc}}(X,x))$ of $\pi_1^{\mathrm{loc}}(X,x)$, but in order to define an invariant of the singularity and not only of the local fundamental group, we will use  only special or cohomologycally special representations (see Example~\ref{ex:G1.G2} below). The following definition is a natural consequence of this remark.
\begin{definition}\label{def:top.spec}
Let $(X,x)$ be the germ of a rational surface singularity and $L$ its link. The \emph{topological spectrum of $(X,x)$} is the following set
\begin{align*}
    \mathrm{T Spf}(X)&=\{0\neq\widehat{c}_{\alpha,1}\in H^1(L;\CC/\ZZ)\, | \, \alpha \in \text{Irr}(\pi_1^{\mathrm{loc}}(X,x)), \text{$\alpha$ is special} \}.
\end{align*}
Moreover, if $(X,x)$ is a quotient singularity, the \emph{topological spectrum of $(X,x)$} is the following set
\begin{align*}
    \mathrm{T Spf}(X)&=\{\widehat{c}_{\alpha,1}\in H^1(L;\CC/\ZZ)\, | \, \alpha \in \text{Irr}(\pi_1^{\mathrm{loc}}(X,x)), \text{$\alpha$ is special and $\widehat{c}_1(\alpha)\neq 0$ } \}\\
    &\bigcup \{\widehat{c}_{\alpha_{\mathrm{Nat}},2}  \, | \, \text{$\alpha_{\mathrm{Nat}}$ is the natural representation}\}.
\end{align*}  
If $(X,x)$ is a Gorenstein singularity, the \emph{topological spectrum of $(X,x)$} is the following set
\begin{align*}
    \mathrm{T Spf}(X)&=\{0\neq\widehat{c}_{\alpha,1}\in H^1(L;\CC/\ZZ)\, | \, \alpha \in \text{Irr}(\pi_1^{\mathrm{loc}}(X,x)), \text{$\alpha$ is cohomologically special} \}.
\end{align*}
\end{definition}

The following example shows that the topological spectrum of a singularity does depend on the singularity, not only on the local fundamental group.
\begin{example}\label{ex:G1.G2}
    Consider the following cyclic subgroups of order $3$:
    \begin{equation*}
        \Gamma_1= \left\langle \left( \begin{matrix} \zeta_3& 0\\ 0& \zeta_3 \end{matrix}\right)  \right\rangle \subset \GLn{2}
        \quad \text{and} \quad 
        \Gamma_2=\left \langle \left(\begin{matrix} \zeta_3& 0\\ 0& \zeta_3^2 \end{matrix}\right) \right\rangle \subset \SLn{2}.
    \end{equation*} 
    Denote by
    \begin{equation*}
        (X_1,x_1) := (\CC^2 / \Gamma_1,0) \quad \text{and} \quad (X_2,x_2) := (\CC^2 / \Gamma_2,0),
    \end{equation*}
    the corresponding quotient singularities.
    The germ $(X_2,x_2)$ is a rational double point singularity. Let $L_j$ be the link of the singularity $(X_j,x_j)$ for $j=1,2$. In both cases, the links are $3$-dimensional lens spaces. Moreover, $\pi_1(L_j)\cong \mathsf{C}_3$ for $j=1,2$. Therefore $\pi_1(L_1)$ and $\pi_1(L_2)$ have the same irreducible non-trivial representations, namely $\alpha_2$  and $\alpha_3$. By Artin and Verdier \cite{Artin-Verdier:RMORDP} and Wunram~\cite{zbMATH03997956} we have that:
    \begin{itemize}
        \item for the singularity $(X_1,x_1)$ there is only one non-trivial special irreducible representation and,
        \item for the singularity $(X_2,x_2)$ both  $\alpha_2$  and $\alpha_3$ are special representations.
    \end{itemize} 
    Therefore, $\mathrm{T Spf}(X_1)$ has only one cohomology class but $ \mathrm{T Spf}(X_2)$ has two cohomology classes. Hence, even if $\pi_1(L_1)$ and $\pi_1(L_2)$ have the same representation theory, the topological spectrums of $(X_1,x_1)$ and $(X_2,x_2)$ are indeed different. They depend on the singularity.
\end{example}

By Theorem~\ref{th:C1C2SonSpf} and Remark~\ref{remark:specialRDP} for rational double point singularities $X$ we can recover some
spectral numbers from the topological spectrum $\mathrm{T Spf}(X)$ evaluating in particular homology classes: $\widehat{c}_{\alpha,1}(\bar{c})$ and $\widehat{c}_{\alpha_{\mathrm{Nat}},2}([L])$, and taking appropriate multiples of the obtained numbers.

\begin{remark}
Let $(\CC^2/\Gamma,0)$ be a quotient singularity ($\Gamma\subset\GLn{2}$).
By Wunram~\cite{zbMATH03997956}, the coefficients $n_1, \dots, n_s$ of the fundamental cycle on the minimal resolution
$\Rs_{\mathrm{min}}$ of $X$ coincide with the dimensions of the non-trivial special irreducible representations $\alpha_1, \dots, \alpha_s$ of $\Gamma$.
Thus, for quotient singularities the multiplicative factor $\frac{1+\sum n_i^2}{ 1+\sum n_i}$ that appears in $\Xi_{\rho}(\Rs_{\mathrm{min}},D)$ (Definition~\ref{def:Xi.Fund}) for the minimal resolution
$\Rs_{\mathrm{min}}$ depends on the group (dimensions of irreducible representations) and on the singularity (only the special ones).

In the case of a rational double point singularity $(X,x)\cong(\CC^2/\Gamma,0)$ ($\Gamma\subset\SLn{2}$), all the non-trivial irreducible representations are special (Remark~\ref{remark:specialRDP}) and
it is well-known that the order $|\Gamma|$ of $\Gamma$ is given by $|\Gamma|  = 1+\sum_{i=1}^s \rank \alpha_i^2$.
Therefore,
\begin{equation*}
\frac{1+\sum_{i=1}^s n_i^2}{1+\sum_{i=1}^s n_i} =\frac{|\Gamma|}{1+\sum_{i=1}^s \rank \alpha_i}.
\end{equation*}
By Remark~\ref{rk:Xi.xi}, in this case one has $\Xi_{\alpha_{\mathrm{Nat}}}(X,\Rs)=\left(\frac{1+\sum n_i^2}{ 1+\sum n_i}\right)\tilde{\xi}_{\alpha_{\mathrm{Nat}}}(D)$,
and by Table~\ref{tb:red.xi} $\tilde{\xi}_{\alpha_{\mathrm{Nat}}}(D)=\frac{1}{|\Gamma|}$, for $\Gamma=\mathsf{BD}_{2r},\mathsf{BT},\mathsf{BO},\mathsf{BI}$. Therefore, in these cases one has
\begin{equation*}
\Xi_{\alpha_{\mathrm{Nat}}}(X,\Rs)=\frac{1}{1+\sum_{i=1}^s \rank \alpha_i}.
\end{equation*}

\end{remark}

\subsection{The invariant \texorpdfstring{$\tilde{\xi}_{\rho}(D)$}{xip(D)} via the McKay correspondence}
In this final section we show how to compute the invariant $\tilde{\xi}_{\rho}(D)$ using the McKay correspondence. The advantage of this approach is that it is more readily computable.

Let $(X,x)$ be the germ of a normal surface singularity. Denote by $L$ the link. Suppose that $L$ is a rational homology sphere, e.g., the link of any rational singularity (see Remark~\ref{remark:linkrhs}). Let $\pi \colon \Rs \to X$ be a resolution such that $E$ is a normal crossing divisor. Let $(L,\sigma_\text{can})$ be the link equipped with the canonical spin$^c$ structure $\sigma_\text{can}$ induced by the resolution (see~\cite[\S~2]{zbMATH01987653}). Moreover, the almost complex structure on $\Rs$ gives a spin$^c$ structure $\sigma_{\Rs}$ on $\Rs$ such that its restriction to $L$ is $\sigma_\text{can}$. From now on, we will use this spin$^c$ structure.

\begin{theorem}
\label{th:C2Res}
Let $(X,x)$ be the germ of a normal surface singularity. Suppose that the link $L$ is a rational homology sphere. Let $\pi \colon \Rs \to X$ be a resolution with $E$ a normal crossing divisor. Let $\rho \colon \pi_1(L) \to GL(n,\CC)$ be a representation. Denote by $M_\rho$ the reflexive $\Ss{X}$-module given by Riemman-Hilbert correspondence and $\rho$. Denote by $\mathscr{M}_\rho$ the corresponding vector bundle. Then,
\begin{equation}\label{eq:xi.integral}
\tilde{\xi}_{\rho}(D)= \int_{\Rs} \left( \mathrm{ch} \, \mathscr{M}_\rho -n \right) \mathcal{T}(\Rs),
\end{equation}
where $\mathrm{ch}$ is the Chern character and $\mathcal{T}$ the Todd class. Moreover, if $\rho$ is topologically trivial, then $\widehat{c}_{\rho,2}([L])$ is given by \eqref{eq:xi.integral}.
\end{theorem}
\begin{proof}
The proof follows by \cite[pp.~415]{zbMATH03491931} and Corollary~\ref{cor:c2.xi}.
\end{proof}
\begin{corollary}
Let $(X,x)$ be the germ of a quotient surface singularity. Let $\pi \colon \Rs \to X$ be the minimal resolution. Denote by $E=\bigcup E_j$ the exceptional divisor. Then, there exists a well-defined natural map
\begin{equation*}
    \Psi \colon \{E_1, \dots, E_n\} \to \QQ/\ZZ.
\end{equation*}
\end{corollary}
\begin{proof}
By~\cite{zbMATH03997956}, there exists a one-to-one correspondence between the irreducible components of the exceptional divisor and the non-trivial special representations of the local fundamental group. Thus, the map is given by Theorem~\ref{th:C2Res} (recall that any quotient singularity is a rational singularity).
\end{proof}

We can obtain a similar result if $(X,x)$ has a smoothing with Milnor fiber $F$.  In the case of a hypersurface one can take $F$ to be the usual Milnor fiber.
Then, $F$ has a spin$^c$ structure whose restriction to the link coincides with $\sigma_\text{can}$ (see~\cite[\S~2]{zbMATH01987653}). Thus,
\begin{theorem}
\label{th:C2Res1}
Let $(X,x)$ be the germ of a normal surface singularity. Suppose that the link $L$ is a rational homology sphere. Let $\pi \colon \Rs \to X$ be a resolution with $E$ a normal crossing divisor. Suppose that $(X,x)$ has a smoothing with Milnor fiber $F$. Let $L$ be the link of $(X,x)$. Let $\rho \colon \pi_1(L) \to GL(n,\CC)$ be a representation and $V_\rho$ its associated flat vector bundle over $L$.
Suppose that there exist an extension $\mathscr{V}_\rho$ of $V_\rho$ over $F$. Then,
\begin{equation}\label{eq:xi.integral.smoothing}
 \tilde{\xi}_{\rho}(D) = \int_{F} \left( ch(\mathscr{V}_\rho) -n \right) \mathcal{T}(F).
\end{equation}
Moreover, if $\rho$ is topologically trivial, then $\widehat{c}_{\rho,2}([L])$ is given by \eqref{eq:xi.integral.smoothing}.
\end{theorem}
\begin{proof}
The proof follows by \cite[pp.~415]{zbMATH03491931} and Corollary~\ref{cor:c2.xi}.
\end{proof}

\subsection*{Acknowledgments}

This research was supported by the project UNAM-DGAPA-PAPIIT-IN105121 and CONACYT 282937. We thank Prof. Javier F. de Bobadilla, Prof. András Nemethi and Prof. Duco van Straten for 
their interest in this problem and helpful comments. The third author is funded by OTKA 126683 and Lend\"ulet 30001. The third author thanks CIRM, Luminy, for its hospitality and for providing a perfect work environment. He also thanks Prof. Javier F. de Bobadilla, the 2021 semester 2 Jean-Morlet Chair, for the invitation.


\end{document}